\documentclass[12pt]{article}
\usepackage{amsmath, amsthm, amssymb}
\usepackage{mathrsfs}
\usepackage{enumerate}
\usepackage{hyperref}
\usepackage{fancyhdr}
\usepackage[bb=boondox]{mathalfa}
\usepackage{geometry}
\title{Auslander–Yorke Dichotomy and Its Generalizations for Non-Autonomous Dynamical Systems}

\author{%
  Saksham Malik\thanks{Department of Mathematics, University of Delhi, Delhi-110007, India.
    \newline\href{mailto:sakshamalik@gmail.com}{sakshamalik@gmail.com}}%
  \and
  Mohammad Salman\footnote{Corresponding author} \thanks{Shyama Prasad Mukherji College for Women, Department of Mathematics,  University of Delhi, Delhi-110026, India.
    \newline\href{mailto:salman25july@gmail.com}{salman25july@gmail.com}}%
  \and
  Ruchi Das\thanks{Department of Mathematics, University of Delhi, Delhi-110007, India.
    \newline\href{mailto:rdasmsu@gmail.com}{rdasmsu@gmail.com}}%
}

\date{}

\begin{document}
\maketitle
\begin{abstract}
    We investigate the dynamics of periodic non-autonomous discrete dynamical systems on uniform spaces and topological spaces, focusing on the extension of the classical Auslander–Yorke dichotomy to these settings. We prove various dichotomy theorems in the uniform-space framework, showing that a minimal periodic non-autonomous system is either sensitive or equicontinuous, and prove some more refined versions involving syndetic equicontinuity and thick sensitivity and eventual sensitivity versus equicontinuity on compact uniform spaces. We further introduce topological analogues like topological equicontinuity, Hausdorff sensitivity, and their syndetic and multi-sensitive variants and prove corresponding Auslander–Yorke–type dichotomies on \(T_3\) spaces.
\end{abstract}

\newtheorem{theorem}{Theorem}[section]
\newtheorem{lemma}[theorem]{Lemma}
\newtheorem{proposition}[theorem]{Proposition}
\newtheorem{corollary}[theorem]{Corollary}
\theoremstyle{definition}
\newtheorem{remark}[theorem]{Remark}
\newtheorem{definition}[theorem]{Definition}
\newtheorem{example}[theorem]{Example}

\newenvironment{proofsketch}
{\noindent\textbf{Proof:}}{\hfill$\square$\\[2mm]}
\newenvironment{construction}
  {\begin{proof}[Construction]}
  {\end{proof}}

\section{Introduction}
Dynamical systems theory provides a powerful framework to model and predict time-varying phenomena across various fields like epidemiology, climate science, economics and engineering\cite{SkiadasSkiadas2016,smith_thieme_2011,strogatz2018nonlinear}. In 1996, Kolyada and Snoha \cite{kolyada1996topological} introduced non-autonomous discrete dynamical systems. Non-autonomous dynamical systems are systems where the evolution depends on the current state as well as time and are therefore ideal for modeling real-world phenomena influenced by time-varying external forces or inputs. Non-autonomous dynamical systems have been widely applied in life sciences, population dynamics, and ecology \cite{caraballo_han_2016,KloedenPoeitzsche2013,  smith_thieme_2011} and in recent years, they have gained considerable traction in machine learning and deep learning \cite{NEURIPS2018_69386f6b,li_lin_shen_2023}.

 Properties of non-autonomous dynamical systems like chaos, stronger forms of sensitivity, topological transitivity, topological mixing, equicontinuity, entropy, etc. have been extensively explored in metric setting \cite{miralles2018sensitive,raghav2018finitely,salman2019dynamics,zhu2016devaney}. Recently, these properties have been explored in the more general  setting of uniform spaces \cite{Salman2021Sensitivity, shao2023chaos, wu2019rigidity}. In this paper, we consider a non-autonomous discrete dynamical system defined on a topological space $X$ without isolated points. Let $ \{f_n \}_{n \in \mathbb N}$ be a sequence of continuous functions $ f_n: X \rightarrow X, n \in \mathbb N$. The system evolves according to the difference equation:
\[
x_{n+1} = f_n(x_n), \quad n \in \mathbb{N}. \tag{1}
\]
Here, the right-hand side depends explicitly on the current time step. The trajectory of a point \(x \in X\) is obtained by successively applying the time-dependent maps \(f_n\). If \(f_n = f\) for all \(n \geq 1\), equation (1) reduces to the classical autonomous discrete dynamical system \((X, f)\).

The Auslander–Yorke dichotomy, introduced by Joseph Auslander and James A. Yorke in \cite{auslander1980interval}, is an important result in classical autonomous dynamical systems which shows that  minimal systems are either sensitive or equicontinuous. There have been several generalizations and extensions of Auslander-Yorke dichotomy for autonomous systems. Auslander, Akin and Berg \cite{akin1997transitive} proved a similar dichotomy for transitive systems showing that transitive autonomous systems are either sensitive or almost equicontinuous, and Huang, Kolyada and Zhang \cite{huang2018analogues} proved a dichotomy for minimal autonomous systems based on thick sensitivity and syndetic equicontinuity. Some more Auslander-Yorke dichotomy type theorems are obtained in \cite{ye2016sensitivity} involving \(\mathcal{F}\)-sensitivity and \(\mathcal{F}\)-equicontinuity for various Furstenberg families \(\mathcal{F}\). Recently, there have been efforts to generalize uniform properties of dynamical systems like sensitivity and equicontinuity to the more general topological space setting through properties like Hausdorff sensitivity, topological sensitivity and topological equicontinuity \cite{good2020equicontinuity,good2018what}. Some Auslander-Yorke dichotomy type theorems based on these topological generalizations are obtained for autonomous systems in \cite{good2020equicontinuity}.

Due to inherent complexities in general non-autonomous dynamical systems, researchers have focused on specific subclasses, particularly periodic non-autonomous dynamical systems to study properties like minimality,  and various forms of sensitivity \cite{raghav2018finitely,Salman2021Sensitivity}. Though considerable progress has been made on this front, transitive points have received comparatively less attention despite its important connections with various forms of sensitivity and equicontinuity. This motivated us to investigate the structure of transitive points in periodic non-autonomous systems to deepen our understanding of equicontinuity and sensitivity and thus establish various Auslander–Yorke dichotomy type theorems. Periodic non-autonomous dynamical systems effectively capture real-world processes with intrinsic periodicity and are therefore widely employed in climate forecasting, mechanical systems under cyclic forcing, seasonal financial time series, and seasonal population dynamics. The results derived in this paper may have direct applications in these fields.

This paper is organized into 4 sections. In Section 2, we present basic definitions and concepts used throughout the paper. In Section 3, we show that a periodic non-autonomous dynamical system has the same transitive points as its induced autonomous system under the condition that the set of transitive points of the induced autonomous system is non-empty. Under the same condition, we show that the uniform space analogues of the Auslander-Yorke dichotomy hold (Theorem \ref{T3}, Theorem \ref{T6} and Theorem \ref{ES}). We provide examples to show that without the aforementioned condition the above theorems need not hold. We also construct an example to show that even on a connected space $X$, it is possible that a periodic non-autonomous dynamical system $(X,f_{1,\infty})$ is minimal while its induced autonomous dynamical system $(X,g)$ has no transitive points, disproving the result \cite[Proposition 2]{raghav2018finitely} which claimed that a periodic non-autonomous system $(X,f_{1,\infty})$ is minimal if and only if its induced autonomous system $(X,g)$ is minimal, given $X$ is connected. In Section 4, we define topological equicontinuity, syndetic topological equicontinuity, Hausdorff sensitivity, thick Hausdorff sensitivity and multi-Hausdorff sensitivity for non-autonomous dynamical systems. Analogues of the Auslander-Yorke dichotomy type theorems proved in Section 3 are shown to hold in the more general setting of topological spaces (Theorem \ref{TAY1} and Theorem \ref{TAY2}).   

\section{Preliminaries}
In this section, we present some definitions that are required for the rest of the paper. Let $\mathbb N$ denote the set of natural numbers and $\mathbb Z^+=\{0\} \cup \mathbb N$. Throughout the paper $X$ is assumed to be a topological space without isolated points, and additional conditions on the space $X$ are explicitly stated whenever required. A topological space $X$ is called regular if for every closed set $F \subseteq X$ and every point $x \in X$ with $x \notin F$, there exist non-empty open sets $U,V$ of $X$ such that $x \in U, 
\ 
F \subseteq V, 
\
\text{and} 
\ 
U \cap V = \emptyset$. A regular Hausdorff space is called a $T_3$ space. For $x\in X$, let $\mathcal{N}_x$ denote the neighborhood basis of $x$. Let us denote $f_{1,\infty} := \{f_n\}_{n=1}^{\infty}$ and $f_i^k=f_{i+k-1} \circ \cdots \circ f_{i+1} \ \circ f_i$. A non-autonomous discrete system $(X, f_{1,\infty})$ is said to be periodic if there exists an $l \in \mathbb{N}$ such that $f_{n+lm}(x) = f_n(x)$, for all $x \in X,\ m \in \mathbb{N}$, and $1 \leq n \leq l$ and the least such $l$ is said to be the period of $(X, f_{1,\infty})$. For a periodic non-autonomous discrete system $(X, f_{1,\infty})$ with period $p$, the induced autonomous system is $(X, g)$, where $g:=f^{\,p}_{1}=f_{p}\circ f_{p-1}\circ\cdots\circ f_{1}$. For any $x \in X$, the orbit of $x$, denoted by $\mathcal{O}_{f_{1,\infty}}(x)$, is the set $\mathcal{O}_{f_{1,\infty}}(x) := \left\{ f^n_1(x) : n \in \mathbb Z^+ \right\}$, where  $f^0_1 = \text{id}$ is the identity map on $X$. A point $x\in X$ is called a transitive point if $\overline{\mathcal{O}_{f_{1,\infty}}(x)}=X$ and the set of transitive points of $X$ is denoted by $ \mathrm{Trans}(X,f_{1,\infty})$. The $\omega$-limit set of $x \in X$ is defined as $\omega_{f_{1,\infty}}(x) := \bigcap_{N \in \mathbb{N}} \overline{ \{ f^n_1(x) : n > N \} }$. We denote the non-wandering set of $x\in X$ by $\Omega_{f_{1,\infty}}(x)$ and $y \in \Omega_{f_{1,\infty}}(x)$ if and only if for any $U \in \mathcal{N}_x$,  $V \in \mathcal{N}_y$ and $N \in \mathbb{N}$, there exist an $n>N$ such that $f_1^n(U)\cap V\neq \emptyset$.\\

\begin{definition}\cite{kelley1955general}
A uniform structure on a set $X$ is a non-empty collection $\mathcal{U}$ of subsets of $X \times X$ (called \emph{entourages}) satisfying:

\begin{enumerate}
  \item $\mathcal{U} \subseteq \mathcal{D}_X$, where $\mathcal{D}_X$ denotes all subsets of $X \times X$ containing the diagonal $\Delta_X = \{(x, x) : x \in X\}$;
  \item If $D_1 \in \mathcal{U}$ and $D_1 \subseteq D_2 \in \mathcal{D}_X$, then $D_2 \in \mathcal{U}$;
  \item For all $D_1, D_2 \in \mathcal{U}$, we have $D_1 \cap D_2 \in \mathcal{U}$;
  \item For every $D \in \mathcal{U}$, there exists an $E \in \mathcal{U}$ such that $E \circ E \subseteq D$, where $E \circ E := \{(x, z) \in X \times X : \text{there exists a } y \in X \text{ with } (x, y) \in E \text{ and } (y, z) \in E \}$.
\end{enumerate}

The pair $(X, \mathcal{U})$ is called a \emph{uniform space}. The \emph{uniform topology} $|\mathcal{U}|$ on $X$ is generated by basic neighborhoods of the form $D[x] := \{ y \in X : (x, y) \in D \}$, for $D \in \mathcal{U}$. A uniform space $(X, \mathcal{U})$ is said to be Hausdorff if and only if the intersection of all entourages is the diagonal i.e., $\bigcap \mathcal{U} = \Delta_X$.
\end{definition}

\begin{definition}\cite{tian2006chaos}
Let $(X, f_{1,\infty})$ be a non-autonomous dynamical system. Then  $(X, f_{1,\infty})$ is topologically transitive if, for every pair of non-empty open sets $U,V$ of $X$, there exists an $n \in \mathbb{N}$ such that $f^n_1(U) \cap V \neq \emptyset$.
\end{definition}

\begin{definition}\cite{Willard1970}
Let $(X, \mathcal{U})$ and $(Y, \mathcal{V})$ be uniform spaces, and let $\mathcal{F} $ be a family of maps from $X$ to $Y$. The family $\mathcal{F}$ is said to be \emph{equicontinuous} at a point $x \in X$ if for every entourage $D \in \mathcal{V}$, there exists an entourage $E \in \mathcal{U}$ such that for all $y \in X$ with $(x, y) \in E$ and for all $f \in \mathcal{F}$, we have $(f(x), f(y)) \in D.$ The family is \emph{equicontinuous on $X$} if it is equicontinuous at every point $x \in X$.
\end{definition}

\begin{definition}\cite{LuChen2017}
\leavevmode
\begin{itemize}
  \item A set $T \subseteq \mathbb{N}$ is \emph{thick} if for all $ k \in \mathbb{N}$, there exists an  $ n \in \mathbb{N}$ such that $\{n, n+1, \dots, n+k\} \subseteq T$.

  \item A set $S \subseteq \mathbb{N}$ is \emph{syndetic} if there exists an $M \in \mathbb{N}$ such that for all $ n \in \mathbb{N}$, $\{n,n+1,..., n+M\} \cap S \neq \emptyset$.
\end{itemize}
\end{definition}

\begin{definition}\cite{Salman2021Sensitivity}
Let $(X, f_{1,\infty})$ be a non-autonomous dynamical system on a uniform space $(X, \mathcal{U})$. For an entourage $D \subseteq X \times X$ and a subset $U \subseteq X$, define 
\[N_{f_{1,\infty}}(U, D) := \left\{ n \in \mathbb{N} : \text{there are }\, x, y \in U \text{ satisfying } (f_1^{n}(x), f_1^{n}(y)) \notin D \right\}.\] Then:
\begin{enumerate}
  \item The system is \emph{sensitive} if there exists an entourage $D \in \mathcal{U}$ (called a \emph{sensitive entourage}) such that for every non-empty open set $U \subseteq X$, we have $N_{f_{1,\infty}}(U, D) \neq \emptyset$.
  \item It is \emph{syndetically sensitive} if there exists an entourage $D \in \mathcal{U}$ such that for every non-empty open set $U \subseteq X$, the set $N_{f_{1,\infty}}(U, D)$ is \emph{syndetic}.
  \item It is \emph{multi-sensitive} if there exists an entourage $D \in \mathcal{U}$ such that for any $k \in \mathbb{N}$ and any collection of non-empty open sets $U_1, \dots, U_k \subseteq X$, we have $\bigcap_{i=1}^k N_{f_{1,\infty}}(U_i, D) \neq \emptyset.$
  
\end{enumerate}
\end{definition}

\begin{definition}\cite{salman2025topological}
Let $(X, f_{1,\infty})$ be a non-autonomous dynamical system, and let $\mathcal{U}$ be an open cover of $X$. For any subset $V \subseteq X$, define:
\[
N_{f_{1,\infty}}(V, \mathcal{U}) := \left\{ n \in \mathbb{N} : \text{there are }\, u, v \in V \text{ satisfying } (u, v) \notin \bigcup_{U \in \mathcal{U}} f_1^{-n}(U) \times f_1^{-n}(U) \right\}.
\]

Then:
\begin{enumerate}
  \item $(X, f_{1,\infty})$ is called \emph{topologically sensitive} if there exists an open cover $\mathcal{U}$ of $X$, called a \emph{sensitivity cover} (\emph{s-cover}),  such that for every non-empty open set $V$ of $X$, $N_{f_{1,\infty}}(V, \mathcal{U}) \neq \emptyset.$ 

  \item $(X, f_{1,\infty})$ is called \emph{syndetically topologically sensitive} if there exists an open cover $\mathcal{U}$ of $X$, called a \emph{syndetic sensitivity cover} (\emph{ss-cover}), such that for every non-empty open set $V$ of $X$, $N_{f_{1,\infty}}(V, \mathcal{U})$ is syndetic.
  
  \item $(X, f_{1,\infty})$ is called \emph{thickly  topologically sensitive} if there exists an open cover $\mathcal{U}$ of $X$, called a \emph{thick sensitivity cover} (\emph{ts-cover}), such that for every non-empty open set $V$ of $X$, $N_{f_{1,\infty}}(V, \mathcal{U})$ is thick.

  \item $(X, f_{1,\infty})$ is called \emph{multi-topologically sensitive} if there exists an open cover $\mathcal{U}$ of $X$, called a \emph{multi-sensitivity cover} (\emph{ms-cover}), such that for every collection of non-empty open sets $V_1, \dots, V_m$ of $X$, we have $\bigcap_{i=1}^m N_{f_{1,\infty}}(V_i, \mathcal{U}) \neq \emptyset.$
  
\end{enumerate}
\end{definition}

\section{Auslander-Yorke Dichotomy and its Analogues on Uniform Spaces.}

In this section, we introduce the notion of syndetic equicontinuity and eventual sensitivity for non-autonomous dynamical systems. We show that under some conditions various forms of Auslander-Yorke dichotomy hold for a periodic non-autonomous dynamical system, characterizing minimal periodic non-autonomous system in terms of properties like sensitivity and equicontinuity, thick sensitivity and syndetic equicontinuity, etc. We also present examples to show that these conditions are important.

\begin{definition}
Let \(\bigl(X, f_{1,\infty}\bigr)\) be a non-autonomous dynamical system on a uniform space \((X,\mathcal U)\).  \(\bigl(X, f_{1,\infty}\bigr)\)  is called \emph{syndetically equicontinuous} at a point $ x\in X$, if for every entourage \(E\in\mathcal U\) there exists a neighborhood $U$ of $x$ such that \(J_{f_{1,\infty}}(U,E)=\{n \in \mathbb N :\text{for all } x_1,x_2\in U , 
\bigl(f_1^n(x_1),\,f_1^n(x_2)\bigr)\;\in\;E\}\) is syndetic.\\
We denote by \(\mathrm{Eq}_{syn}\bigl(X,f_{1,\infty}\bigr)\) the set of all syndetically equicontinuous points.  The system \((X,f_{1,\infty})\) is \emph{syndetically equicontinuous} if
\(\mathrm{Eq}_{syn}(X,f_{1,\infty})=X\).
\end{definition}

\begin{remark}
    Clearly, if \(\bigl(X, f_{1,\infty}\bigr)\) is  equicontinuous then \(\bigl(X, f_{1,\infty}\bigr)\) is syndetically equicontinuous. In Example \ref{E4}, we show that the Sturmian subshif $(X_\alpha,\sigma)$ is syndetically equicontinuous but not equicontinuous which shows that the syndetic equicontinuity need not imply equicontinuity.
\end{remark}

\begin{definition}
Let \(\bigl(X,f_{1,\infty}\bigr)\) be a non-autonomous dynamical system on a uniform space \((X,\mathcal U)\).  We say that the system \(\bigl(X,f_{1,\infty}\bigr)\) is \emph{eventually sensitive} if there exists an \emph{eventual-sensitivity entourage} \(D\in\mathcal U\) such that
for all $x\in X$ and  $E\in\mathcal U$, there exist
$n,k\in\mathbb N$ and $y\in E\bigl[f_1^n(x)\bigr]$
 with $
\bigl(f_1^{\,n+k}(x),\,f_{1}^k(y)\bigr)\;\notin\;D.$
\end{definition}
\begin{remark}
    Clearly if \(\bigl(X, f_{1,\infty}\bigr)\) is sensitive then it is eventually sensitive. But eventual sensitivity need not imply sensitivity \cite[Example 5.2]{good2020equicontinuity}.
\end{remark}

The following theorem characterizes the transitive points of a periodic non-autonomous system whenever its induced autonomous map has a nonempty transitive set, providing the essential link needed to extend Auslander-Yorke dichotomy and its variations into the non-autonomous setting.

\begin{theorem}\label{T1}
    let \((X,f_{1,\infty})\) be a periodic non‑autonomous system such that \(\mathrm{Trans}(X,g)\neq\emptyset\), then $\mathrm{Trans}(X,g)\;=\;\mathrm{Trans}(X,f_{1,\infty})$.
\end{theorem}
\begin{proof}
 Let \emph{p} be the period of \((X,f_{1,\infty})\), then \(g:=f^{\,p}_{1}=f_{p}\circ f_{p-1}\circ\cdots\circ f_{1}\).
 Clearly, \(\mathrm{Trans}(X, g)\subseteq\mathrm{Trans}(X,f_{1,\infty})\)
as \(\mathcal{O}_ g(y)\subseteq\mathcal{O}_{f_{1,\infty}}(y)\), for each $y \in X$.  Let \(x\in\mathrm{Trans}(X,f_{1,\infty})\), then for each \(n\in \mathbb N\), writing \(n=p\,k+r\) with \(0\le r<p\), we have
\[
f_1^n(x)
= f_1^r\bigl(f_1^{p\,k}(x)\bigr)
= f_1^r\bigl( g^k(x)\bigr) 
\;\in\;
f_1^r(\mathcal{O}_ g(x)) \subseteq f_1^r(\overline{\mathcal{O}_ g(x)}), 
\]
and hence
\[
\mathcal{O}_{f_{1,\infty}}(x)
\;\subseteq\;
\bigcup_{r=0}^{p-1}{f_1^r\bigl(\mathcal{O}_ g(x)\bigr)}.
\]
Since \(x\in\mathrm{Trans}(X,f_{1,\infty})\), 
\[
X
\;=\;
\overline{\mathcal{O}_{f_{1,\infty}}(x)}
\;\ = \;
\bigcup_{r=0}^{p-1}\overline{f_1^r\bigl(\mathcal{O}_g(x)\bigr)}
.
\]
Set \(A=\mathcal{O}_ g(x)\), then \(\overline{A}\) is non-empty and closed.  

\medskip\noindent
\emph{Case 1: \(\mathrm{int}(\overline{A})\neq\emptyset\).}  Since \(\mathrm{Trans}(X, g)\neq\emptyset\), \(\mathrm{Trans}(X, g)\) is dense in X implying there exists \(z\in \overline{A}\ \cap \mathrm{Trans}(X, g) \). Therefore
\(\overline{\mathcal{O}_g(z)}=X\subseteq \overline{A}\), forcing \(\overline{A}=X\). Hence \(x\in\mathrm{Trans}(X, g)\).

\medskip\noindent
\emph{Case 2: \(\mathrm{int}(\overline{A})=\emptyset\).}  Then \(A\) is nowhere dense.  If possible assume that $\overline{f_1^r(A)}$ is nowhere dense for $r \in \{0,1,....,p-1\}$. Since $\overline{f_1^{p-1}(A)}$ is nowhere dense, $X\setminus\overline{f_1^{p-1}(A)}$ is dense and contained in $\bigcup_{r=0}^{p-2}\overline{f_1^r(A)}$ as $X\;=\;\bigcup_{r=0}^{p-1}\overline{f_1^r(A)}$ so we have 
\[X\;=\;\bigcup_{r=0}^{p-2}\overline{f_1^r(A)}.\] Continuing similarly we get $X=\overline{A}$, which is a contradiction as \emph{A} is nowhere dense.\\
So, there exists an $r \in \{0,1,....,p-1\}$, such that $\overline{f_1^r(A)}\neq\emptyset$. Let $U \subseteq X$ be a non-empty open set with $U\;\subseteq\;\operatorname{int}\bigl(\overline{f_{1}^{r}(A)}\bigr).$\\
By continuity of $f_{1}^{r}$, there is a non-empty open set $V\subseteq X$ such that $f_{1}^{r}(V)\;\subseteq\;U\;\subseteq\;\overline{f_{1}^{r}(A)}.$\\ Consequently, we get that
\[
g(V)
\;=\;
\bigl(f_{p}\circ\cdots\circ f_{r+1}\bigr)\bigl(f_{1}^{r}(V)\bigr)
\;\subseteq\;
\bigl(f_{p}\circ\cdots\circ f_{r+1}\bigr)\bigl(\overline{f_{1}^{r}(A)}\bigr)
\;\subseteq\;\overline{\bigl(f_{p}\circ\cdots\circ f_{1}\bigr)\bigl(A\bigr)}=
\overline{\,g(A)\,}
\;\subseteq\;
\overline{A}
\]
and $g^2(V) \;=\; g\bigl(g(V)\bigr)
\;\subseteq\;
g\bigl(\overline{A}\bigr)
\;\subseteq\;
\overline{g(A)}
\;\subseteq\;
A$. Hence for any $n\in \mathbb N$, we have $g^n(V)\subseteq \overline{A}$, but since \emph{V} is open in X, there exists $w\in V\cap\mathrm{Trans}(X,g)$ which is a contradiction as $\mathcal{O}_g(w)\subseteq \ \bigcup_{i=0}^\infty g^i(V)\subseteq \overline{A}$.\\
So we conclude that \(\mathrm{int}(\overline{A})\neq\emptyset\) and \(\overline{A}=X\), implying \(\overline{\mathcal{O}_g(x)}=X\) and \(x\in\mathrm{Trans}(X,g)\). Thus \(\mathrm{Trans}(X,f_{1,\infty})\subseteq\mathrm{Trans}(X,g)\), and equality follows. 
 
\end{proof}

\begin{example}\label{E1}
    If $X$ has isolated points then there exists a periodic non-autonomous system $(X,f_{1,\infty})$ such that $
\mathrm{Trans}(X,g)\subsetneq \mathrm{Trans}(X,f_{1,\infty})$ as constructed below.
\end{example}

Let $X=S^1 \cup  \{ (2,0), (3,0)\}$, where $S^1=\{e^{i\theta}: \theta \in [0,2\pi)\}$. We define $f_1: X \rightarrow X $ as follows  
\[f_1(x)=\begin{cases}
    e^{i(\theta+\alpha/2)}, & x=e^{i\theta} \in S^1,\\
    (3,0), & x=(2,0),\\
    (2,0), & x=(3,0).
\end{cases}\] 
and $f_2:X \rightarrow X$ as follows
\[f_2(x)=\begin{cases}
    e^{i(\theta+\alpha/2)}, & x=e^{i\theta} \in S^1,\\
    (1,0), & x=(2,0),\\
    (3,0), & x=(3,0).
\end{cases}\]
where $\alpha$ is an irrational multiple of $\pi$.\\
Clearly $f_1$, $f_2$ are continuous. Now, consider the 2-periodic non-autonomous dynamical system $(X,f_{1,\infty})$ generated by $\{f_1,f_2\}$, then $(X, g)$, where $g=f_2 \circ f_1$, is the induced autonomous dynamical system.\\
Observe that, $g|_{S^1}: e^{i\theta} \rightarrow e^{i(\theta +\alpha)}$ is the irrational rotation by $\alpha$, and $g((2,0))=(3,0)$ and $g((3,0))=(1,0)$. Now, $\overline{(\mathcal{O}_g((1,0)))}=S^1$, we know that $g((2,0))=(3,0)$ and $g^2((2,0))=(1,0)$ implying that $\overline{(\mathcal{O}_g((1,0)))}\subseteq \overline{(\mathcal{O}_g((2,0)))}$ and therefore $\overline{(\mathcal{O}_g((2,0)))}= S^1 \cup \{(2,0), (3,0)\}=X$ which implies that $(2,0)\in \mathrm{Trans}(X,g) $. \\Also,  $\overline{(\mathcal{O}_g((3,0)))}=S^1 \cup \{(3,0)\}$ implying $(3,0) \not \in \mathrm{Trans}(X,g) $. But we know that $f_1((3,0))=(2,0)$ and $\overline{(\mathcal{O}_g((3,0)))} \subseteq \overline{(\mathcal{O}_{f_{1,\infty}}((3,0)))}, $ so $\overline{(\mathcal{O}_{f_{1,\infty}}((3,0)))}=S^1 \cup \{(2,0), (3,0)\}=X $, hence $(3,0) \in \mathrm{Trans}(X,f_{1,\infty}) $. Thus, we have $
\mathrm{Trans}(X,g)\subsetneq \mathrm{Trans}(X,f_{1,\infty})$.

\begin{example}\label{E2}
There exists a periodic non-autonomous dynamical system \((X, f_{1,\infty})\) such that \((X, f_{1,\infty})\) is minimal but the induced autonomous dynamical system \((X, g)\) is not minimal, in fact $\mathrm{Trans}(X,g)=\emptyset$ as discussed below.
\end{example}

Define
$A=\{e^{i\theta}:\theta\in[0,2\pi)\},\;B=\{2+e^{i\theta}:\theta\in[0,2\pi)\}$ $,\;X=A\cup B$   with the usual subspace topology of $\mathbb C,\;p=1=e^{i0}=2+e^{i\pi}$ $,\;\alpha/\pi\notin\mathbb{Q},\;g_1:A\to B,\;g_1(e^{i\theta})=2+e^{i(\theta+\alpha)},\;g_2:B\to B,\;g_2(2+e^{i\theta})=2+e^{i(\theta+\alpha+\pi)}$
and paste them to obtain
\[
f_1 \colon X \to X, \qquad
f_1(x) = 
\begin{cases}
g_1(x), & x \in A,\\
g_2(x), & x \in B.
\end{cases}
\]
Similarly, define
$g_3: A \to A,\; g_3(e^{i\theta}) = e^{i(\theta+\pi)},\quad\text{and}\quad g_4: B \to A,\; g_4(2 + e^{i\theta}) = e^{i\theta}$
and paste to obtain
\[
f_2 \colon X \to X, \qquad
f_2(x) = 
\begin{cases}
g_3(x), & x \in A,\\
g_4(x), & x \in B.
\end{cases}
\]

By construction, \(g_1(p) = g_2(p)\) and \(g_3(p) = g_4(p)\), so \(f_1\) and \(f_2\) are continuous on \(X\). Let \((X, f_{1,\infty})\) be the 2-periodic non-autonomous system generated by \(\{f_1, f_2\}\). Consider the induced autonomous system $(X,g)$,  where $g=f_2 \circ f_1$.

We now show that \((X, f_{1,\infty})\) is minimal but \((X, g)\) is not.

\medskip
\noindent\textbf{Case 1:} Let \(x \in A\), so \(x = e^{i\theta}\) for some \(\theta \in [0, 2\pi)\). Then
\[
g \big|_A \colon e^{i\theta} \mapsto e^{i(\theta+\alpha)},
\]
is an irrational rotation on \(A\) and the orbit \(\mathcal{O}_{g}(x) = \{g^m(x) : m \ge 0\}\) is dense in \(A\). Let \(U \subseteq B\) be any open arc, so that $U = \{2 + e^{i\phi} : \phi \in (\phi_1, \phi_2) , \ \  \phi_1, \phi_2 \in [0, 2\pi)\}$, then $U' = \{e^{i\psi} : \psi \in (\phi_1 - \alpha, \phi_2 - \alpha)\} \subseteq A$ is an open arc in $A$.

By density of \(\mathcal{O}_{g}(x)\), there exists an \(m \in \mathbb{N}\) such that \(g^m(x) \in U'\), hence
\[
f_1(g^m(x)) = f_1((f_2 \circ f_1)^m(x)) = f_1^{2m+1}(x) \in U.
\]
Since \(U\) was arbitrary in \(B\), \(\mathcal{O}_{f_{1,\infty}}(x)\) is dense in both \(A\) and \(B\).

\medskip
\noindent\textbf{Case 2:} Let \(x \in B\), so \(x = 2 + e^{i\theta}\). Set \(y = f_1(x) = 2 + e^{i(\theta + \alpha + \pi)}\). Consider the following map on \(B\)
\[
h = (f_1 \circ f_2) \big|_B \colon 2 + e^{i\psi} \mapsto 2 + e^{i(\psi + \alpha)},
\]
which is again an irrational rotation with orbit \(\mathcal{O}_h(y)\) dense in \(B\). We have the following subcases: 

\medskip
\noindent\textbf{(a) Visiting an arc in \(A\):} Let $W = \{e^{i\phi} : \phi \in (\phi_3, \phi_4), \ \phi_3, \phi_4 \in [0, 2\pi) \} \subseteq A$ be any arc in $A$, then $W' = \{2 + e^{i\phi} : \phi \in (\phi_3, \phi_4)\} \subseteq B $ is an arc in $B$. By minimality of $h = (f_1 \circ f_2) \big|_B$, there exists an \(n\in \mathbb N\) such that \(h^n(y) \in W'\), then
\[
f_2(h^n(y)) \in W \quad \Rightarrow \quad f_2 \circ (f_1 \circ f_2)^n (f_1(x)) = (f_2 \circ f_1)^{n+1}(x) = f_1^{2n+2}(x) \in W.
\]

\medskip
\noindent\textbf{(b) Visiting an arc in \(B\):} Let $V = \{2 + e^{i\phi} : \phi \in (\phi_5, \phi_6), \ \phi_5, \phi_6 \in [0,2\pi)\} \subseteq B.$ By density of \(\mathcal{O}_h(x)\), there exists a \(k\in \mathbb N\) such that \(h^k(f_1(x)) \in V\). Then
\[
f_1^{2k+1}(x) = f_1 \circ (f_2 \circ f_1)^k(x) = h^k(f_1(x)) \in V.
\]

\medskip
\noindent In both subcases, the orbit of \(x \in B\) meets every open set in \(X = A \cup B\), proving minimality.

\medskip
\noindent However, note that \(g(X) \subseteq A\), so the autonomous system \((X, g)\) is not minimal.

\begin{remark}
    Example \ref{E2} shows that even on a connected space $X$ it is possible for a periodic non-autonomous system $(X,f_{1,\infty})$ to be minimal while the induced autonomous system is not minimal. This also disproves result \cite[Proposition 2]{raghav2018finitely}.
\end{remark}

Theorem \ref{T3} extends the classical Auslander–Yorke theorem to periodic non-autonomous dynamical systems.  

\begin{theorem}\label{T3}
Let \(\bigl(X,f_{1,\infty}\bigr)\) be a periodic non-autonomous dynamical system on a Hausdorff uniform space $(X, \mathcal{U})$ such that $\mathrm{Trans}(X,g)\neq\emptyset$. If
\( \ 
\mathrm{Eq}\bigl(X,f_{1,\infty}\bigr)\neq\emptyset  ,
\)
then,
\(
\mathrm{Trans}\bigl(X,f_{1,\infty}\bigr)\;\subseteq\;\mathrm{Eq}\bigl(X,f_{1,\infty}\bigr).
\)
Further, if \( \bigl(X,f_{1,\infty}\bigr)\) is minimal then it is either sensitive or equicontinuous.
\end{theorem}

\begin{proof}
Let $p$ be the period of \((X,f_{1,\infty})\), then \(g:=f^{\,p}_{1}=f_{p}\circ f_{p-1}\circ\cdots\circ f_{1}\). Suppose \(x\in \mathrm{Trans}\bigl(X,f_{1,\infty}\bigr)\) and $E \in \mathcal{U}$.  Choose a symmetric entourage \(E'\) with \(E'\circ E'\subseteq E\).  For \(z\in\mathrm{Eq}(X,f_{1,\infty})\), there is a symmetric entourage \(D'\) such that
\[
f_1^n\bigl(D'[z]\bigr)\;\subseteq\;E'\bigl[f_1^n(z)\bigr],
\
\text{for each }n\in\mathbb{N}.
\]
  By Theorem \ref{T1}, \(\mathrm{Trans}(X,g)=\mathrm{Trans}(X,f_{1,\infty})
\) so \( \ x \in \mathrm{Trans}\bigl(X,g\bigr)\), by continuity of \(g\) there exist an \(N\in\mathbb{N}\) and a symmetric entourage \(D''\subseteq D'\) with
\[
f_1^{Np}\bigl(D''[x]\bigr)
=
g^N\bigl(D''[x]\bigr)
\;\subseteq\;
D'[z].
\]
Then for every \(j\ge0\) and \(y\in D''[x]\),
\[
f_1^{Np+j}(y)
= f_1^j\bigl(f_1^{Np}(y)\bigr)
\;\in\;
E'\bigl[f_1^{Np+j}(
z)\bigr],
\]
and since \((f_1^{Np+j}(z),f_1^{Np+j}(x))\in E'\), it follows that \((f_1^{Np+j}(y),f_1^{Np+j}(x))\in E\).\\
For each \( m=\{1,\dots,Np-1 \}\), continuity of \(f_1^m\) at \(x\) gives a symmetric entourage \(D_m\) with $f_1^m\bigl(D_m[x]\bigr)\;\subseteq\;E\bigl[f_1^m(x)\bigr].$
Let
\[
D \;=\; D''\;\cap\;\bigcap_{m=1}^{Np-1}D_m.
\]
Then for every \(k\in\mathbb{N}\) and \(y\in D[x]\), we have \((f_1^k(y),f_1^k(x))\in E\).  Hence,
\(\mathrm{Trans}\bigl(X,f_{1,\infty}\bigr)\subseteq\mathrm{Eq}\bigl(X,f_{1,\infty}\bigr)\). \\
Now, assume that \(\bigl(X,f_{1,\infty}\bigr)\) is minimal. If \(\mathrm{Eq}\bigl(X,f_{1,\infty}\bigr)=\emptyset\) then clearly \(\bigl(X,f_{1,\infty}\bigr)\) is sensitive. If \(\mathrm{Eq}\bigl(X,f_{1,\infty}\bigr)\neq\emptyset\) then \(\mathrm{Trans}\bigl(X,f_{1,\infty}\bigr)\subseteq\mathrm{Eq}\bigl(X,f_{1,\infty}\bigr)\) and since \(\mathrm{Trans}\bigl(X,f_{1,\infty}\bigr)=X\) we have, \(\mathrm{Eq}\bigl(X,f_{1,\infty}\bigr)=X\) and hence \(\bigl(X,f_{1,\infty}\bigr)\) is equicontinuous.
\end{proof}

\begin{remark}
    There exists a periodic non-autonomous dynamical system \(\bigl(X,f_{1,\infty}\bigr)\) with  $\mathrm{Trans}\bigl(X,g\bigr)=\emptyset$, which is minimal but is neither sensitive nor equicontinuous as discussed below.
\end{remark}
\begin{example}\label{E3}
    From \cite{Kurka2003}, we know that $(\mathbb{Z}_2,R_1)$, where $\mathbb Z_2$ is the set of $2$-adic integers, is minimal and equicontinuous and $\mathbb Z_2$ is homeomorphic to the classical middle third Cantor set $\mathcal C$, here $R_1$ is the addition by $e=(1,1,...)$ on $\mathbb Z_2$. From \cite{Kurka2003} and \cite{Moothathu2007} we know that the Sturmian subshift $(X_\alpha,\sigma)$ (where $\alpha$ is an irrational multiple of $\pi$) is minimal and sensitive and $X_\alpha$ is homeomorphic to $\mathcal C$ as well. So, there exists a homeomorphism $h: X_\alpha \rightarrow \mathbb Z_2$. Let $f=h^{-1} \circ R_1 \circ h:X_\alpha \rightarrow X_\alpha $ and hence $h \circ f=R_1 \circ h$. The system $(X_\alpha,f)$ is topologically conjugate to $(\mathbb{Z}_2,R_1)$ and is therefore minimal and equicontinuous.\\ 
    Define, $A=X_\alpha \times \{a\}$, $B=X_\alpha \times \{b\}$, let $X=A \cup B$ and define \[
f_1 \colon X \to X, \qquad
f_1(x) = 
\begin{cases}
(f(x),b), & (x,a) \in A,\\
(\sigma (x),a), & (x,b) \in B.
\end{cases}
\]
\[
f_2 \colon X \to X, \qquad
f_2(x) = 
\begin{cases}
(x,b), & (x,a) \in A,\\
(x,a), & (x,b) \in B.
\end{cases}
\]
Clearly, $f_1$ and $f_2$ are continuous. Let $(X,f_{1,\infty})$ be the 2-periodic non-autonomous dynamical system generated by $\{f_1,f_2\}. $ Define the induced autonomous map $g= f_2 \circ f_1$.
We now show that $(X,f_{1,\infty})$ is minimal.\\
\textbf{Case 1:} $(x,a)\in A$, where $x\in X_\alpha$.
\[ g \big|_A: (x,a) \rightarrow (f(x),a)\]
and so $(A=X_\alpha \times \{a\},g \big|_A) $ is minimal and equicontinuous as $(X_\alpha, f)$ is minimal and equicontinuous. Therefore $\overline{\mathcal{O}_g((x,a))}$ is dense in $A$. Let $U \times \{b\}$ be a non-empty open set in $B$, where $U$ is a non-empty open set of $X_\alpha$. We observe that $f^{-1}(U) \times\{a\} $ is non-empty and open in $A$ and since $\overline{\mathcal{O}_g((x,a))}$ is dense in $A$, there exists an $n \in \mathbb N$ such that, $g^n((x,a)) \in f^{-1}(U)\times \{a\}$. Therefore $f_1^{2n+1}((x,a))=f_1(g^n((x,a)))\in f_1(f^{-1}(U)\times \{a\}) \subseteq U \times \{b\}$. Hence, $\overline{\mathcal{O}_{f_{1,\infty}}((x,a))}$ is dense in $X=A\cup B$.\\
\textbf{ Case 2:} $(x,b) \in B$, where $x \in X_\alpha$. Then
\[ g \big|_B: (x,b) \rightarrow (\sigma(x),b)\]
and $(B=X_\alpha \times \{b\},g \big|_B )$ is minimal and sensitive as $(X_\alpha, \sigma)$ is minimal and sensitive. Therefore $\overline{\mathcal{O}_g((x,b))}$ is dense in $B$. Let $V \times \{a\}$ be a non-empty open set in $A$, where $V$ is a non-empty open set of $X_\alpha$. Then $\sigma^{-1}(V) \times \{b\}$ is non-empty and open in $B$ and since $\overline{\mathcal{O}_g((x,b))}$ is dense in $B$, there exists an $m \in \mathbb N$ such that $g^m((x,b))\in \sigma^{-1}(V) \times \{b\}$. Therefore, $f_1^{2m+1}((x,b))=f_1(g^m((x,b))) \in f_1(\sigma^{-1}(V))\times \{b\}\subseteq V \times \{a\}$. Hence, $\overline{\mathcal{O}_{f_{1,\infty}}((x,b))}$ is dense in $X=A\cup B$ implying $(X,f_{1,\infty})$ is minimal.\\
Clearly $g(A)\subseteq A$ and $g(B)\subseteq B$, so $\mathrm{Trans}(X,g)=\emptyset$. Since $(B, g \big|_B)$ is sensitive, $(X,g)$ is not equicontinuous, and by  \cite[Proposition 1]{raghav2018finitely}  $(X,f_{1,\infty})$ is not equicontinuous. Similarly, since $(A, g \big|_A)$ is equicontinuous,  $(X,g)$ is not sensitive and by \cite[Proposition 4]{raghav2018finitely}, $(X,f_{1,\infty})$ is not sensitive.
\end{example}

\begin{theorem}\label{T4}
Let \(\bigl(X,f_{1,\infty}\bigr)\) be a non-autonomous dynamical system on a Hausdorff uniform space $(X, \mathcal{U})$. We have the following:
\begin{description}
    \item[(1)] If $x \in \mathrm{Eq}\bigl(X,f_{1,\infty}\bigr)$ then $\Omega_{f_{1,\infty}}(x) = \omega_{f_{1,\infty}}(x)$
    \item[(2)] Suppose \(\bigl(X,f_{1,\infty}\bigr)\) is periodic and $\mathrm{Trans}(X,g)\neq\emptyset$. If $\mathrm{Eq}\bigl(X,f_{1,\infty}\bigr)\neq \emptyset$ then \(\mathrm{Trans}\bigl(X,f_{1,\infty}\bigr)=\mathrm{Eq}\bigl(X,f_{1,\infty}\bigr)\)
\end{description}

\end{theorem}

\begin{proof}
\textbf{(1)}  We get the desired result by \cite[Theorem 2.3]{akin1997transitive}.\\ \textbf{(2)} Let $x\in \mathrm{Eq}\bigl(X,f_{1,\infty}\bigr)$. By \cite[Proposition 1]{DegirmenciKocak2003}, $ (X,g) $ is topologically transitive as $\mathrm{Trans}(X,g)\neq\emptyset$, and therefore $(X,f_{1,\infty})$ is topologically transitive. Hence, for  every $x \in X$, $\Omega(x) = X$, and by (1) we get $\omega(x)=\Omega(x)=X$, so $x \in \operatorname{Trans}(X, f_{1,\infty})$ and hence $\operatorname{Eq}(X, f_{1,\infty}) \subseteq \operatorname{Trans}(X, f_{1,\infty})$. By Theorem \ref{T3}, we have \(\mathrm{Trans}\bigl(X,f_{1,\infty}\bigr)\subseteq\mathrm{Eq}\bigl(X,f_{1,\infty}\bigr)\) and hence $\mathrm{Trans}\bigl(X,f_{1,\infty}\bigr)=\mathrm{Eq}\bigl(X,f_{1,\infty}\bigr)$.
\end{proof}

The following theorem extends the analogue of the Auslander–Yorke dichotomy proved by Huang et al. \cite[Theorem 3.4]{huang2018analogues} to the setting of periodic non‐autonomous dynamical systems.

\begin{theorem}\label{T6}
Let \((X,f_{1,\infty})\) be a periodic non-autonomous dynamical system on a Hausdorff uniform space $(X, \mathcal{U})$ such that $\mathrm{Trans}(X,g)\neq\emptyset$. Then either $(X,f_{1,\infty})$ is thickly sensitive ($\mathrm{Eq}_{\text{syn}}(X,f_{1,\infty})= \emptyset$) or $ \mathrm{Trans}(X,f_{1,\infty}) \subseteq \mathrm{Eq}_{\text{syn}}(X,f_{1,\infty})$. In particular, if $(X,f_{1,\infty})$ is minimal then it is either thickly sensitive or syndetically equicontinuous.

\end{theorem}

\begin{proof}
Let $p$ be the period of \((X,f_{1,\infty})\), then \(g:=f^{\,p}_{1}=f_{p}\circ f_{p-1}\circ\cdots\circ f_{1}\). Let $D \in \mathcal{U}$ be a symmetric entourage. Suppose $(X, f_{1,\infty})$ is not thickly sensitive. Then there exists a non-empty open set $U' \subseteq X$ such that the set
\[
S_{f_{1,\infty}}(U', D) = \{ n \in \mathbb{N} : \text{there are } x_1, x_2 \in U' \text{ satisfying } (f_1^n(x_1), f_1^n(x_2)) \notin D \}
\]
is not thick. Thus, its complement $J_{f_{1,\infty}}(U', D) = \mathbb{N} \setminus S_{f_{1,\infty}}(U', D)$ is syndetic implying that for $ u_1, u_2  \in U'$ and $n \in J_{f_{1,\infty}}(U', D)$ we get $(f_1^n(u_1),f_1^n(u_2)) \in D.$
\\
Let \(x\in \mathrm{Trans}(X,f_{1,\infty})\), by Theorem \ref{T1},  $\mathrm{Trans}(X,g)=\mathrm{Trans}(X,f_{1,\infty})$ and hence there exists an $m \in \mathbb{N}$ such that $g^m(x)=f_1^{mp}(x) \in U'$. Since $f_1^{mp}$ is continuous, there exists an open neighborhood $U$ of $x$ such that $f_1^{mp}(U) \subseteq U'$.

Now, for each $y \in U$ and $n \in J_{f_{1,\infty}}(U', D)$, we have
\[
(f_1^{mp+n}(x), f_1^{mp+n}(y)) = (f_1^n(f_1^{mp}(x)), f_1^n(f_1^{mp}(y))) \in D,
\]
as $f_1^{mp}(x), f^{mp}(y) \in U'$, and $n \in J_{f_{1,\infty}}(U', D)$. Hence,
\[
 mp+J_{f_{1,\infty}}(U', D)=\{ mp+n : n \in J_{f_{1,\infty}}(U', D) \}\subseteq J_{f_{1,\infty}}(U, D),
\]
and since a translate of a syndetic set is syndetic, it follows that $J_{f_{1,\infty}}(U, D)$ is syndetic. Thus, $x$ is a syndetically equicontinuous point. Since $x \in \mathrm{Trans}(X, f_{1,\infty})$ was arbitrary, we conclude that $
\mathrm{Trans}(X, f_{1,\infty}) \subseteq \mathrm{Eq}_{\mathrm{syn}}(X, f_{1,\infty}).
$
In particular, if $(X, f_{1,\infty})$ is minimal, then $\mathrm{Trans}(X, f_{1,\infty}) = X$ and so the system is syndetically equicontinuous if it is not thickly sensitive.

\end{proof}
\pagebreak
\begin{remark}
    There exists a periodic non-autonomous dynamical system \(\bigl(X,f_{1,\infty}\bigr)\) with $\mathrm{Trans}\bigl(X,g\bigr)=\emptyset$, which is minimal but is neither thickly sensitive nor syndetically equicontinuous as discussed below:
\end{remark}
\begin{example}\label{E4}
    From \cite{Kurka2003}, we know that the Sturmian subshift $(X_\alpha,\sigma)$ ( $\alpha$ is an irrational multiple of $\pi$ ) is minimal and an almost one-to-one extention of the minimal equicontinuous system $(S^1,R_\alpha)$, where $R_\alpha$ is the irrational rotation by $\alpha$. Hence, by \cite[Proposition 4.2]{huang2018analogues}, $(X_\alpha,\sigma)$ is syndetically equicontinuous. Let $(X_p,\sigma)$ be the minimal and topologically mixing system constructed in \cite{Petersen1970}. Both $X_p$ and $X_\alpha$ are homeomorphic to the Cantor set $ \mathcal{C}$, so there exists a homeomorphism $h: X_\alpha \rightarrow X_p$. Let $f=h^{-1} \circ \sigma \big |_{X_p} \circ h:X_\alpha \rightarrow X_\alpha $ then $h \circ f=\sigma\big|_{X_p} \circ h$. The system $(X_\alpha,f)$ is topologically conjugate to $(X_p,\sigma)$ and is therefore minimal and topologically mixing. By \cite[Theorem 6]{LiuLiaoWang2014} weakly mixing implies thickly sensitive and hence $(X_\alpha, f)$ is thickly sensitive as well.\\ 
    Define, $A=X_\alpha \times \{a\}$, $B=X_\alpha \times \{b\}$, $X=A \cup B$ and define \[
f_1 \colon X \to X, \qquad
f_1(x) = 
\begin{cases}
(f(x),b), & (x,a) \in A,\\
(\sigma (x),a), & (x,b) \in B.
\end{cases}
\]
\[
f_2 \colon X \to X, \qquad
f_2(x) = 
\begin{cases}
(x,b), & (x,a) \in A,\\
(x,a), & (x,b) \in B.
\end{cases}
\]
Clearly $f_1$ and $f_2$ are continuous. Let $(X,f_{1,\infty})$ be the 2-periodic non-autonomous dynamical system generated by $\{f_1,f_2\}. $ Consider the induced autonomous system $(X,g)$ where $g= f_2 \circ f_1$. Similar to the construction in Example \ref{E3} we can show that $(X,f_{1,\infty})$ is minimal and $\mathrm{Trans}(X,g)=\emptyset$.\\
Note that $\bigl(B=X_\alpha \times \{b\}, g \big|_B: (x,b)\rightarrow (f(x),b)\bigr)$ is thickly sensitive as $(X_\alpha , f)$ is thickly sensitive, so $(X,g)$ is not syndetically equicontinuous and by Theorem \ref{SEq}  $(X,f_{1,\infty})$ is not syndetically equicontinuous. Similarly, since $\bigl(A=X_\alpha \times \{a\}, g \big|_A: (x,a)\rightarrow (\sigma(x),a)\bigr)$ is syndetically equicontinuous,  $(X,g)$ is not thickly sensitive and by Theorem \ref{TSen} $(X,f_{1,\infty})$ is not thickly sensitive.\\
\end{example}

The following theorem provides a dichotomy in terms of eventual sensitivity and equicontinuity for periodic non-autonomous dynamical systems. 

\begin{theorem}\label{ES}
Let $(X, \mathcal{U)}$ be a compact Hausdorff uniform space. A periodic  dynamical system $(X, f_{1,\infty})$ with $\mathrm{Trans}(X,g)\neq\emptyset$, is either equicontinuous or eventually sensitive.
\end{theorem}

\begin{proof}
 Since $\mathrm{Trans}(X,g)\neq\emptyset$ by \cite[Proposition 1]{DegirmenciKocak2003}, $(X,g)$
is topologically transitive. By \cite[Theorem 5.3]{good2020equicontinuity},  \((X,g)\) is either equicontinuous or eventually sensitive.

\medskip

\textbf{Case 1:} If \((X,g)\) is  equicontinuous, then by \cite[Proposition 1]{raghav2018finitely}, \((X,f_{1,\infty})\) is also  equicontinuous.

\medskip

\textbf{Case 2:} Suppose \((X,g)\) is eventually sensitive.  
Let \(D_0\in\mathcal U\) be an eventual‑sensitivity entourage for \((X,g)\). For any \(x\in X\) and any entourage \(E\in\mathcal U\), by eventual sensitivity of \((X,g)\), there exist \(q,k\in\mathbb N\) and $y\in E\bigl(g^q(x)\bigr)$ such that 
\(\;(g^{q+k}(x),\,g^k(y))\notin D_0\).  
Since \(f_{1}^{p\,q}(x)=g^q(x)\), and \(f_1^{p(q+k)}(x)=g^{q+k}(x)\), we get that
\[
y\;\in\;E\bigl(f_1^{p\,q}(x)\bigr),
\quad
\bigl(f_1^{p(q+k)}(x),\,f_{1}^{\,p\,k}(y)\bigr)\;\notin\;D_0.
\]
Hence \((X,f_{1,\infty})\) is eventually sensitive with the same entourage \(D_0\).

\end{proof}

Theorem \ref{SEq} establishes that a periodic non-autonomous system is syndetically equicontinuous if and only if its induced autonomous system is syndetically equicontinuous. Similarly, Theorem \ref{TSen} shows that a non-autonomous dynamical system is thickly sensitive if and only if its induced autonomous dynamical system is thickly sensitive.  

\begin{theorem}\label{SEq}
Let \((X,f_{1,\infty})\) be a periodic nonautonomous system on a compact uniform space \((X,\mathcal{U})\). Then \((X,f_{1,\infty})\) is syndetically equicontinuous  if and only if \((X,g)\) is syndetically equicontinuous.
\end{theorem}
\begin{proof}
Suppose first that \((X,f_{1,\infty})\) is syndetically equicontinuous. We show that \((X,g)\) is syndetically equicontinuous.\\
Let \(x\in X\) and \(E\in\mathcal U\) be an entourage.  By uniform continuity of $f_p\circ f_{p-1}\circ\cdots\circ f_{r+1}$, for$\quad r\in \{0,1,\dots,p-1\}$, there is an entourage \(D'\in\mathcal U\) such that
\[
\text{for all }(u,v)\in D'
\; \text{ we have }
\bigl(f_p\circ\cdots\circ f_{r+1}(u),\,f_p\circ\cdots\circ f_{r+1}(v)\bigr)
\in E,\
\text{for each }\,r \in \{0,1,..., p-1\}.
\]

Since \((X,f_{1,\infty})\) is syndetically equicontinuous at \(x\), for this \(D'\) there exists a neighborhood $U$ of $x$ such that \(J_{f_{1,\infty}}(U,D')\) is a syndetic set with bound \(L\).
Define
\[
J_g(U,E) \;=\;
\bigl\{\,n\in\mathbb N : \text{for}\,y_1,y_2\in U,\;(g^n(y_1),\,g^n(y_2))\in E\bigr\}.
\]
We claim that \(J_g(U,E)\) is syndetic with bound \(L+2\).\\
Let \(n_0\in\mathbb N\) be arbitrary.  Consider the block $\{\,n_0p,\;n_0p+1,\;\dots,\;n_0p + (L+1)p\}.$ Since \((L+1)p > L\), the syndeticity of \(J_{f_{1,\infty}}(U,D')\) (bound \(L\)) guarantees that there is some
\[
s\in J_{f_{1,\infty}}(U,D')\;\cap\;\{n_0p,\dots,n_0p+(L+1)p\}.
\]
Write \(s = q\,p + r\) with \(0\le r < p\); then \(n_0 \le q \le n_0 + L + 1\).  By definition of \(J_{f_{1,\infty}}(U,D')\), $\text{for}\,y_1,y_2\in U,\quad
\bigl(f_1^s(y_1),\,f_1^s(y_2)\bigr)\in D'.
$
Applying the  map  \(f_p\circ\cdots\circ f_{r+1}\) and using the fact that if \((u,v)\in D' \text{ then }(f_p\circ\cdots\circ f_{r+1}(u),\,f_p\circ\cdots\circ f_{r+1}(v))\in E\), we get that 
$
\text{for } y_1,y_2 \in U,
(f_p\circ\cdots\circ f_{r+1}(f_1^s(y_1)),\,f_p\circ\cdots\circ f_{r+1}(f_1^s(y_2))
= \bigl( \ g^{q+1}(y_1),g^{q+1}(y_2)\Bigr)
\;\in\;E.
$
Hence $q+1\;\in\;\{n_0,n_0+1,\dots,n_0 + (L+2)\}\;\cap\;J_g(U,E).$ Since \(n_0\) was arbitrary, \(J_g(U,E)\) is syndetic with bound \(L+2\).  This shows that \((X,g)\) is syndetically equicontinuous.\\
Now assume that \((X,g)\) is syndetically equicontinuous. Let \(x\in X\) and \(F\in \mathcal{U}\) be an entourage.  Then there is a neighborhood \(V\) of $x$ such that $J_g(U,F)
\;=\;
\{\,m\in\mathbb N:\text{for }\,y_1,y_2\in V,\;(g^m(y_1),g^m(y_2))\in F\}$ is syndetic. Since $g:=f^{\,p}_{1}=f_{p}\circ f_{p-1}\circ\cdots\circ f_{1}$, $\ p J_g(V,F)=\{pn: n\in J_g(V,F)\} \subseteq J_{f_{1,\infty}}(V,F)$ and hence $J_{f_{1,\infty}}(V,F)$ is syndetic as $p J_g(V,F)$ is syndetic.
\end{proof}
\begin{theorem}\label{TSen}
Let \((X,f_{1,\infty})\) be a periodic nonautonomous system on a compact uniform space \((X,\mathcal{U})\). Then \((X,f_{1,\infty})\) is thickly sensitive if and only if \((X,g)\) is thickly sensitive.
\end{theorem}

\begin{proof}

Suppose \emph{p} is the period of \((X,f_{1,\infty})\), then \(g:=f^{\,p}_{1}=f_{p}\circ f_{p-1}\circ\cdots\circ f_{1}\). We first show that if $(X,g)$ is thickly sensitive then \((X,f_{1,\infty})\) is thickly sensitive.\\
Suppose \((X,g)\) is thickly sensitive. Then there is an entourage \(E_0\in\mathcal U\) such that for every non-empty open set \(U\subseteq X\), the set
\[
N_g(U,E_0)
=\{\,m\ge0 : \text{there exist } x,\ y\in U \text{ satisfying}\;(g^m( x),g^m( y))\notin E_0\}
\]
is thick. By the uniform continuity of each finite composition
\(\;f_p\circ\cdots\circ f_{r+1}\;\text{for }r \in \{0,1,\dots,p-1\}\),
we can choose an entourage \(E'\in\mathcal U\) such that
\[
\text{if }(x,y)\in E'
\text{ then }
\bigl(f_p\circ\cdots\circ f_{r+1}(x),\,f_p\circ\cdots\circ f_{r+1}(y)\bigr)
\in E_0, \text{ for each }r \in \{0,1,..,p-1\}.
\tag{1}
\]
Assume on the  contrary that \((X,f_{1,\infty})\) is not thickly sensitive.  Then for $ E'$, there is a non-empty open set \(U\subseteq X\) such that $N_{f_{1,\infty}}(U,E')$ is not thick and therefore $\mathbb{N}\setminus N_{f_{1,\infty}}(U,E')$ is syndetic which implies that there exists an integer \(L\) such that
\[
\text{for all} \ n\in\mathbb{N},\;\text{we have a}\;j\in\{n,\dots,n+L\} \text{ satisfying} \;(f_1^j(x),f_1^j(y))\in E' \text{ for all } 
\,x,y\in U.
\tag{2}
\]

Since \(N_g(U,E_0)\) is thick, choose \(k\in\mathbb{N}\) with \(k\,p>L\) then there exists an \(n_0\) such  that
\[
\{\,n_0,n_0+1,\dots,n_0+k\}
\;\subset\;
N_g(U,E_0).
\]
Consider the block of iterates
\(\{n_0p,n_0p+1,\dots,n_0p+L\}\).  By (2), there is some
\[
j\in\{n_0p,\dots,n_0p+L\}
\quad\text{such that }
\text{for  }\,x,y\in U,\;(f_1^j(x),f_1^j(y))\in E'.
\]
Write \(j = s\,p + r\) with \(0\le r<p\).  then \(n_0\le s\le n_0+k-1\), which implies 
\[
s+1\;\in\;\{n_0,\dots,n_0+k\}\;\subset\;N_g(U,E_0).
\]
On the other hand, $g^{\,s+1}
= f_p\circ\cdots\circ f_{r+1}\;\circ\;f_1^j$, and by (1) this implies
\[
\text{for all }\,x,y\in U,\quad
\bigl(g^{\,s+1}(x),\,g^{\,s+1}(y)\bigr)
\;\in\;E_0,
\]
contradicting that \(s+1\in N_g(U,E_0)\).  Therefore \((X,f_{1,\infty})\) must be thickly sensitive.

Now assume that \((X,f_{1,\infty})\) is thickly sensitive with sensitivity entourage \(D\in\mathcal{U}\).  We will show that \((X,g)\) is thickly sensitive with the same entourage \(D\).\\Let \(V\subseteq X\) be any non-empty open set and fix an arbitrary \(t\in\mathbb N\). Choose \(k\in\mathbb{N}\) with $k>(t+1)\,p$. Since \((X,f_{1,\infty})\) is thickly sensitive, there exists \(m_0\in\mathbb{N}\) such that
$
\{\,m_0,\;m_0+1,\;\dots,\;m_0+k\}
\;\subseteq\;
N_{f_1,\infty}(V,D)
$. We can write
$
m_0 \;=\; q\,p \;+\; r,
\
0 \;\le\; r < p.
$
Then $\{(q+1)p,\;(q+2)p,\;\dots,\;(q+1+t)p\}\subseteq \{\,m_0,\;m_0+1,\;\dots,\;m_0+k\}$ and therefore
\[
\{(q+1)p,\;(q+2)p,\;\dots,\;(q+1+t)p\}\subseteq N_{f_1,\infty}(V,D).
\]
Hence for each \(j\in \{1,2,\dots,t+1\}\), there exist $x_j,y_j \in V$ such that 
\(\bigl(f_1^{(q+j)p}(x_j),f_1^{(q+j)p}(y_j)\bigr)\notin D\). But$f_1^{\,(q+j)p} \;=\; g^{\,q+j}$, implying $\{\,q+1,\;q+2,\;\dots,\;q+1+t\}
\;\subseteq\;
N_g(V,D)$. Since \(t\) and $V$ were arbitrary, this shows that \((X,g)\) is thickly sensitive with entourage \(D\).
\end{proof}

Based on \cite[Theorem 3.2]{huang2018analogues}, we have the following result:
\begin{theorem}\label{msts}
    Let $(X,f)$ be an autonomous dynamical system on a uniform space $(X, \mathcal{U})$. If $(X,f)$ is multi-sensitive then $(X,f)$ is thickly sensitive. If $\mathrm{Trans}(X,f)\neq\emptyset$ then thick sensitivity of $(X,f)$ implies $(X,f)$ is multi-sensitive.
\end{theorem}

Using Theorem \ref{TSen} and Theorem \ref{msts}, we get the following equivalence between thick sensitivity and multi-sensitivity for periodic non-autonomous dynamical systems.  
\begin{theorem}\label{TMS}
Let \((X,f_{1,\infty})\) be a periodic nonautonomous system on a compact Hausdorff uniform space \((X,\mathcal{U})\) such $\mathrm{Trans}(X,g)\neq\emptyset$. Then \((X,f_{1,\infty})\) is thickly sensitive if and only if \((X,f_{1,\infty})\) is multi-sensitive.
\end{theorem}
\begin{proof}
First assume that \((X,f_{1,\infty})\) is multi-sensitive, then as shown in \cite{Salman2021Sensitivity} $(X,g)$ is multi-sensitive and therefore by Theorem \ref{msts} $(X,g)$ is thickly sensitive and hence by Theorem \ref{TSen}, \((X,f_{1,\infty})\) is thickly sensitive.\\
Now, assume that \((X,f_{1,\infty})\) is thickly sensitive, then by Theorem \ref{TSen}, $(X,g)$ is thickly sensitive and  by Theorem \ref{msts}, $(X,g)$ is multi-sensitive and hence \((X,f_{1,\infty})\) is multi-sensitive.
\end{proof}
\begin{corollary}\label{AYms}
Let \((X,f_{1,\infty})\) be a periodic non-autonomous dynamical system on a compact Hausdorff uniform space $(X, \mathcal{U})$ such that $\mathrm{Trans}(X,g)\neq\emptyset$. If  $(X,f_{1,\infty})$ is minimal then it is either multi-sensitive or syndetically equicontinuous.
\end{corollary}
\begin{proof}
    Suppose \((X,f_{1,\infty})\) is minimal, then by Theorem \ref{T6} \((X,f_{1,\infty})\) is either syndetically equicontinous or thickly sensitive. By Theorem \ref{TMS} \((X,f_{1,\infty})\) is thickly sensitive if and only if it is multi-sensitive.
\end{proof}

\section{Auslander-Yorke Dichotomy type Theorems on Topological Spaces}

In this section, we introduce the notions of topological equicontinuity, syndetic topological equicontinuity, Hausdorff sensitivity, multi-Hausdorff sensitivity, syndetic Hausdorff sensitivity, and thick Hausdorff sensitivity for non-autonomous dynamical systems. We derive various Auslander-Yorke dichotomy type theorems for periodic non-autonomous dynamical systems on $T_3$ spaces.

\begin{definition}
Let \((X,f_{1,\infty})\) be a non-autonomous dynamical system on a topological space \(X\).  For \(x,y\in X\) we say \((x,y)\) is an \emph{equicontinuity pair}, and write $(x,y)\in \mathrm{EqP}\bigl(X,f_{1,\infty}\bigr)$,
if for every neighborhood \(O\in N_y\), there exist neighborhoods \(U\in N_x\) and \(V\in N_y\) such that for each, $n\in\mathbb N$, if $f_1^n(U)\,\cap\,V\neq\emptyset$ then $f_1^n(U)\subseteq O$. The system is called \emph{topologically equicontinuous} at \(x\in X\)  if
\(
\{x\}\times X\;\subseteq\;\mathrm{EqP}\bigl(X,f_{1,\infty}\bigr),
\)
and the system itself is said to be \emph{topologically equicontinuous} if
\(\mathrm{EqP}(X,f_{1,\infty})=X\times X\).  We denote the set of topological equicontinuity points by
\[
\mathrm{TEq}\bigl(X,f_{1,\infty}\bigr)
=\{\,x\in X: \{x\} \times X\subseteq \mathrm{EqP}(X,f_{1,\infty}) \}.
\]
\end{definition}

\begin{definition}
Let \((X,f_{1,\infty})\) be a non-autonomous dynamical system on a topological space \(X\).  For \(x,y\in X\) we say that \((x,y)\) is a \emph{syndetic equicontinuity pair}, and write $(x,y)\in \mathrm{SEqP}\bigl(X,f_{1,\infty}\bigr)$,
if for every \(O\in N_y\) there exist  \(U\in N_x\) and \(V\in N_y\) such that $
J_{f_{1,\infty}}(U,V;O)
=\bigl\{\,n\in\mathbb N : \text{ if } f_1^n(U)\cap V\neq\emptyset\; \text{then}\; f_1^n(U)\subseteq O\bigr\}
$ is syndetic.\\
The system  is called \emph{syndetically topologically equicontinuous} at \(x\in X\)  if
\(
\{x\}\times X\;\subseteq\;\mathrm{SEqP}\bigl(X,f_{1,\infty}\bigr),
\)
and the system itself is said to be \emph{syndetically topologically equicontinuous} if
\(\mathrm{SEqP}(X,f_{1,\infty})=X\times X\).  We denote the set of syndetic topological equicontinuity points by
\[
\mathrm{STEq}\bigl(X,f_{1,\infty}\bigr)
=\{\,x\in X: (x,y)\in \mathrm{EqP}(X,f_{1,\infty}),\ \text{for each }y\in X\}.
\]
\end{definition}
\begin{remark}
    It has been shown that on a compact Uniform space, topological equicontinuity and equicontinuity are equivalent \cite{royden1988real}, similarly one can show that syndetic topological equicontinuity and syndetic equicontinuity are equivalent on a compact space.
\end{remark}

\begin{definition}
    Let $(X, f_{1,\infty})$ be a non-autonomous dynamical system on a Hausdorff space $X$, and let $\mathcal{U}$ be an open cover of $X$. For any subset $V \subseteq X$, define:
\[
N_{f_{1,\infty}}(V, \mathcal{U}) := \left\{ n \in \mathbb{N} : \text{there are }\, u, v \in V \text{ satisfying } (u, v) \notin \bigcup_{U \in \mathcal{U}} f_1^{-n}(U) \times f_1^{-n}(U) \right\}.
\]

Then:
\begin{enumerate}
  \item $(X, f_{1,\infty})$ is called \emph{Hausdorff sensitive} if there exists a finite open cover $\mathcal{U}$ of $X$ such that for every non-empty open set $V \subseteq X$, $N_{f_{1,\infty}}(V, \mathcal{U}) \neq \emptyset.$ 

  \item $(X, f_{1,\infty})$ is called \emph{syndetically Hausdorff sensitive} if there exists a finite open cover $\mathcal{U}$ of $X$ such that for every non-empty open set $V \subseteq X$, $N_{f_{1,\infty}}(V, \mathcal{U})$ is syndetic.
  
  \item $(X, f_{1,\infty})$ is called \emph{thickly  Hausdorff sensitive} if there exists a finite open cover $\mathcal{U}$ of $X$ such that for every non-empty open set $V \subseteq X$, $N_{f_{1,\infty}}(V, \mathcal{U})$ is thick.

  \item $(X, f_{1,\infty})$ is called \emph{multi-Hausdorff sensitive} if there exists a finite open cover $\mathcal{U}$ of $X$ such that for every collection of non-empty open sets $V_1, \dots, V_m \subseteq X$, we have $\bigcap_{i=1}^m N_{f_{1,\infty}}(V_i, \mathcal{U}) \neq \emptyset.$
  
\end{enumerate}
\end{definition}

\begin{remark}
    It has been shown in \cite[Theorem 3.1]{salman2025topological} that for compact metric spaces, the notions of topological sensitivity and sensitivity are equivalent. It is further shown that similar relation can be shown for multi-topological sensitivity, syndetic topological sensitivity and thick topological sensitivity. Also, \cite[Theorem 3.2]{good2018what} shows that for a compact Hausdorff uniform space, the notions of Hausdorff sensitivity and sensitivity are equivalent for autonomous systems. Based on these results it can be shown that a similar relation holds for Hausdorff sensitivity, syndetic Hausdorff sensitivity, thick Hausdorff sensitivity and multi-Hausdorff sensitivity for non-autonomous dynamical systems.
\end{remark}

\begin{definition}
    Let $(X,f_{1,\infty})$ be a non-autonomous dynamical system on a Hausdorff space X. If $(x,y)\not\in \mathrm{TEq}\bigl(X,f_{1,\infty}\bigr)$ then there exists an $O\in\mathcal N_y$ such that for all $U\in \mathcal N_x$ and $V\in \mathcal N_y$, there exists an $n \in \mathbb N$ such that $f_1^n(U) \cap V\neq \emptyset $  and $f_1^n(U) \not \subseteq O$. Then $O\in \mathcal N_y$ is called a \emph{splitting neighborhood of $y$ with respect to $x$.}
\end{definition}

With  similar arguments, we have the following two results for non-autonomous dynamical systems from \cite{good2020equicontinuity} :

\begin{lemma}\label{CTE}
Let $(X, f_{1,\infty})$ be a non-autonomous dynamical system on a Hausdorff space $X$, 
 Then for any $x, y \in X$, any $n \in \mathbb{N}$, and for any neighborhood $O$ of $y$, there exist neighborhoods $U$ of $x$ and $V$ of $y$ such that for every $k \in \{1, 2, \ldots, n\}$, if $f_1^k(U) \cap V \neq \emptyset$ then $f_1^k(U) \subseteq O$.
\end{lemma}

\begin{lemma}\label{ls}
Let \((X,f_{1,\infty})\) be a non‐autonomous system on a Hausdorff space \(X\).  Suppose \((x,y)\notin \mathrm{EqP}\bigl(X,f_{1,\infty}\bigr)\) and \(O\) is a splitting neighbourhood of \(y\) with respect to \(x\).  Then for \(U \in \mathcal N_x\) and \( V \in \mathcal N_y \), the set $\bigl\{\,n\in\mathbb{N} : f_{1}^{n}(U)\cap V \neq\emptyset\text{ and }f_{1}^{n}(U)\not\subseteq O \bigr\}$ is infinite.
\end{lemma}

We now build some results to get a topological analogue of Auslander-Yorke dichotomy for periodic non-autonomous dynamical systems. 

\begin{theorem}\label{TEQ}
    Let \((X,f_{1,\infty})\) be a periodic non-autonomous dynamical system on a Hausdorff space $X$ such that \(\mathrm{Tran}(X,g)\neq\emptyset\). If $ \mathrm{TEq}(X,f_{1,\infty})\neq\emptyset $ then \(\mathrm{Tran}(X,f_{1,\infty})\subseteq \mathrm{TEq}(X,f_{1,\infty})\). Further, if \((X,f_{1,\infty})\) is minimal then it is topologically equicontinuous.
\end{theorem}
\begin{proof}
Assume that \(\mathrm{TEq}(X,f_{1,\infty})\neq\emptyset\).  Let $x \in \mathrm{Trans}(X,f_{1,\infty})$ and $z \in \mathrm{TEq}(X,f_{1,\infty})$ and fix any \(y \in X\) and neighborhood \(O\) of \(y\).  Since \((z,y) \in \mathrm{EqP}(X,f_{1,\infty})\), there exist non-empty open sets $V',\ W \text{ of } X \text{ with } z \in W, \ 
 y \in V'$ such that $\text{for each},\ n\in\mathbb N$, if $
f_1^n(W)\cap V'\neq\emptyset$ then $f_1^n(W)\subseteq O.$\\
By Theorem \ref{T1}, \(x \in \mathrm{Trans}(X,g)=\mathrm{Trans}(X,f_{1,\infty})\). Therefore there is an \(N\in\mathbb N\) with $g^N(x) = f_1^{Np}(x) \in W.$\\
By Continuity of \(g^N\) there exists \(W' \subseteq X\) with $x \in W'\text{ such that}\ g^N(W') \subseteq W.$ Then for each \(j\ge0\),
\[
f_1^{Np+j}(W')
= f_1^j\bigl(f_1^{Np}(W')\bigr)
= f_1^j\bigl(g^N(W')\bigr)
\subseteq f_1^j(W),
\]
so whenever \(f_1^{Np+j}(W')\cap V\neq\emptyset\) we have \(f_1^j(W)\cap V\neq\emptyset\) and hence \(f_1^{Np+j}(W')\subseteq O\).

By Lemma \ref{CTE}, for  \(m\in \{1,2,\dots,Np-1\}\) there exist open sets $U'$ and $V''$,  $\text{ of } X \text{ with } x \in U', \ y \in V''$ such that if $f_1^m(U')\cap V''\neq\emptyset$
then $f_1^m(U')\subseteq O.$

Define $U \;=\; W' \;\cap\;U' \text{ and } V \;=\; V' \;\cap\;V''.$ Then \(x\in U\), \(y\in V\), and for any \(n\in\mathbb N\), if $f_1^n(U)\cap V\neq\emptyset$
then $f_1^n(U)\subseteq O.$ So \((x,y)\in \mathrm{EqP}(X,f_{1,\infty})\).  Since \(y\) and \(O\) were arbitrary, \(x \in \mathrm{TEq}(X,f_{1,\infty})\).  Therefore $\mathrm{Trans}(X,f_{1,\infty}) \;\subseteq\; \mathrm{TEq}(X,f_{1,\infty}).$
If \((X,f_{1,\infty})\) is minimal, then \(\mathrm{Trans}(X,f_{1,\infty})=X\), implying that \(\mathrm{TEq}(X,f_{1,\infty})=X\), i.e., the system is topologically equicontinuous.
\end{proof}

\begin{theorem}\label{T21}
Let \((X,f_{1,\infty})\) be a periodic non-autonomous dynamical system on a $T_3$ space $X$ such that \(\mathrm{Tran}(X,g)\neq\emptyset\). Suppose that \(x\in \mathrm{Trans}(X,f_{1,\infty})\). If there exists a \(y\in X\) with $(x,y)\notin \operatorname{EqP}(X,f_{1,\infty})$, then the system \((X,f_{1,\infty})\) is Hausdorff sensitive.
\end{theorem}

\begin{proof}
Suppose \emph{p} is the period of \((X,f_{1,\infty})\), then \(g:=f^{\,p}_{1}=f_{p}\circ f_{p-1}\circ\cdots\circ f_{1}\). Let \(x\) and \(y\) be as in the hypothesis. Then by Theorem \ref{T1}, $x\in \mathrm{Trans}(X,g)$, and by hypothesis, there exists a splitting neighborhood $O$ of $y$ with respect to $x$. Since \(X\) is $T_3$, we can choose open neighborhoods \(V_1\) and \(V_2\) of \(y\) satisfying $\overline{V}_1\subseteq O \ \text{and} \ \overline{V}_2\subseteq V_1.$

Define the finite open cover $\mathcal{U}=\{V_1,\, X\setminus \overline{V_2}\}.$ Now, let \(W\subseteq X\) be an arbitrary non-empty open set. Since \(x\in \mathrm{Trans}(X,g)\) therefore,  there exists an integer \(n\in\mathbb{N}\) such that $x\in W'=g^{-n}(W).$

By Lemma \ref{ls}, for \(W'\in \mathcal N_x\)  and \(V_1 \in \mathcal N_y\), there exists an integer \(m>pn\) with $f_1^m(W')\cap V_2\neq\emptyset \quad\text{and}\quad f_1^m(W')\not\subseteq O$. Since $W' = g^{-n}(W)=(f_{1}^{pn})^{-1}(W) \text{ we get } f_1^m(W') = f_1^{m-pn}(W)$. For \(k=m-pn\), we have $f_1^k(W)\cap V_2\neq\emptyset \quad\text{and}\quad f_1^k(W)\not\subseteq O.$\\
In particular, there exist points \(a,b\in W\) such that $f_1^k(a)\notin O \quad \text{and} \quad f_1^k(b)\in V_2.$ Thus, $\{f_1^k(a),f_1^k(b)\}\cap V_1=\{f_1^k(b)\} \quad \text{and} \quad \{f_1^k(a),f_1^k(b)\}\cap (X\setminus \overline{V_2})=\{f_1^k(a)\}.$

Thus the images of \(a\) and \(b\) lie in different members of the cover \(\mathcal{U}\). Since \(W\) was arbitrary, it follows by definition that \((X,f_{1,\infty})\) is Hausdorff sensitive.
\end{proof}

The following theorem is the topological analogue of the Auslander-Yorke dichotomy for periodic non-autonomous dynamical systems.

\begin{theorem}\label{TAY1}
Let \((X,f_{1,\infty})\) be a periodic non-autonomous dynamical system on a $T_3$ space $X$ such that \(\mathrm{Tran}(X,g)\neq\emptyset\). If \((X,f_{1,\infty})\) is minimal then it is either topologically equicontinuous or Hausdorff sensitive.
\end{theorem}

\begin{proof}
Suppose \((X,f_{1,\infty})\) is minimal. There are now two possibilities:

\medskip
\noindent\textbf{Case 1:} \(\mathrm{TEq}(X,f_{1,\infty})\neq \emptyset\). Then by Theorem \ref{TEQ}, \(\mathrm{TEq}(X,f_{1,\infty})=X\). Hence, the system \((X,f_{1,\infty})\) is topologically equicontinuous.

\medskip
\noindent\textbf{Case 2:} Suppose no point is topologically equicontinuous. In this case, for every \(x\in X\) there exists some \(y\in X\) such that $(x,y)\notin \operatorname{EqP}(X,f_{1,\infty})$ and by the Theorem \ref{T21}, the system \((X,f_{1,\infty})\) is Hausdorff sensitive.

\medskip
Thus, a minimal non-autonomous dynamical system \((X,f_{1,\infty})\) must be either topologically equicontinuous or Hausdorff sensitive.
\end{proof}

With similar arguments, we have the following two results for non-autonomous dynamical systems from \cite{good2020equicontinuity} :

\begin{lemma}\label{LTE}
Let \((X, f_{1,\infty})\) be a non-autonomous dynamical system on a Hausdorff space \(X\),  Suppose that \((x,y) \in \operatorname{EqP}(X, f_{1,\infty})\). Then either $y\notin \Omega_{f_{1,\infty}}(x)$  or else \(y\in\omega_{f_{1,\infty}}(x)\).
\end{lemma}

\begin{lemma}\label{LTE2}
Let \((X, f_{1,\infty})\) be a non-autonomous dynamical system on a Hausdorff space \(X\), if \((X, f_{1,\infty})\) is topologically equicontinuous at \(x\). Then, $\omega_{f_{1,\infty}}(x) = \Omega_{f_{1,\infty}}(x).$
\end{lemma}

\begin{proposition}
    
    Let \((X,f_{1,\infty})\) be a periodic non-autonomous dynamical system on a Hausdorff space $X$ such that \(\mathrm{Tran}(X,g)\neq\emptyset\). If $ \mathrm{TEq}(X,f_{1,\infty)}\neq\emptyset $ then \(\mathrm{Tran}(X,f_{1,\infty})= \mathrm{TEq}(X,f_{1,\infty})\).
\end{proposition}
\begin{proof}
  By Theorem \ref{TEQ} we have \(\mathrm{Tran}(X,f_{1,\infty})\subseteq \mathrm{TEq}(X,f_{1,\infty})\). Let $x\in \mathrm{TEq}(X,f_{1,\infty})$, then by Lemma \ref{LTE2}, $\Omega_{f_{1,\infty}}(x)=\omega_{f_{1,\infty}}(x)$ and by transitivity of \((X,f_{1,\infty})\), we get that $\Omega_{f_{1,\infty}}(x)=X$, so $\omega_{f_{1,\infty}}(x)=X$. Hence  $x\in \mathrm{Trans}(X,f_{1,\infty})$ and therefore \(\mathrm{TEq}(X,f_{1,\infty})\subseteq \mathrm{Trans}(X,f_{1,\infty})\) and \(\mathrm{Tran}(X,f_{1,\infty})= \mathrm{TEq}(X,f_{1,\infty})\).
\end{proof}

The following theorem provides sufficient conditions under which a periodic non-autonomous dynamical system is topologically equicontinuous if and only its induced autonomous system is topologically equicontinuous.

\begin{theorem}\label{thm:periodic-nonauto-tec}
Let \((X,f_{1,\infty})\) be a periodic non-autonomous dynamical system on a Hausdorff space $X$ and let \((X,g)\) be its induced autonomous system. If each $f_i$ is feebly open and bijective, then $(X,f_{1,\infty})$ is topologically equicontinuous if and only if $(X,g)$ is topologically equicontinuous.

\end{theorem}

\begin{proof} 
Suppose \emph{p} is the period of \((X,f_{1,\infty})\), then \(g:=f^{\,p}_{1}=f_{p}\circ f_{p-1}\circ\cdots\circ f_{1}\). Assume \((X,f_{1,\infty})\) is topologically equicontinuous.  Then for all \(x,y\in X\) and any neighborhood \(O\) of \(y\), there are neighborhoods \(U\in \mathcal{N}_x\) and \(V\in \mathcal N_y\) such that if $f_1^n(U)\cap V\neq\emptyset$ then $f_1^n(U)\subseteq O$ for each $n\in\mathbb N.$ In particular, take \(n=m\,p\), $m\in \mathbb N$.  Since 
\(
f_1^{mp}
= (f_p\circ\cdots\circ f_1)^m
= g^m,
\)
we have 
\[
\text{if }g^m(U)\cap V\neq\emptyset
\;\text{then}\;
g^m(U)\subseteq O, \ 
\text{ for every }\,m\in\mathbb N,
\]
so \((X,g)\) is topologically equicontinuous.

\medskip

Now assume that \((X,g)\) is topologically equicontinuous.  Fix arbitrary \(x,y\in X\) and a neighborhood \(O \in \mathcal N_y\).  We will produce neighborhoods $U$ of $x$, $V$ of $y$ such that if $f_1^n(U)\cap V\neq\emptyset$ then $f_1^n(U)\subseteq O$, for each \(n\in \mathbb N\).

By topological equicontinuity of \((X,g)\), there are neighborhoods $U_0$ of $x$ and $V_0$ of $y$ such that if $g^q(U_0)\cap V_0\neq\emptyset$ then $g^q(U_0)\subseteq O$, for each $q\in\mathbb N.$

For \(r \in \{1,\dots,p-1\}\), ler $w_r\in X$ be such that  \(f_1^r(w_r)=y\).  Then choose a neighborhood $V_r$ of $w_r$ such that
    \(\;f_1^r(V_r)\subseteq O.\)
    By topological equicontinuity of \((X,g)\), there exist neighborhoods $U_r$ of x and $W_r$ of $w_r$ such that
    \[
    \text{if }g^q(U_r)\cap W_r\neq\emptyset
    \;\text{then}\;
    g^q(U_r)\subseteq V_r, 
    \text{ for each }\,q\in\mathbb N.
    \]

Set
    \[
      U \;=\;\bigcap_{r=0}^{p-1}U_r,
      \qquad
      V \;=\; V_0\;\cap\;\bigcap_{r=1}^{p-1}Int\bigl(f_1^r(W_r)\bigr).
    \]
    Note that each \(Int(f_1^r(W_r))\) is non-empty by feeble openness of each $f_i$, so \(V\) is a neighborhood of \(y\).

Now write any \(n\in\mathbb N\) uniquely as \(n=p\,q + r\) with \(0\le r\le p-1\).  
    - If \(r=0\), then
      \[
        f_1^n(U)
        = f_1^{pq}(U)
        = g^{\,q}(U).
      \]
      Hence \(f_1^n(U)\cap V\neq\emptyset\) implies \(g^q(U)\cap V_0\neq\emptyset\), whence \(g^q(U)\subseteq O\) which implies that \ \(f_1^n(U)\subseteq O\).

     If \(1\le r\le p-1\), then
      \[
        f_1^n \;=\; f_1^{pq + r}
        \;=\;
        f_1^r\circ g^q.
      \]
      If
      $
        f_1^n(U)\cap V
        =
        \bigl(f_1^r\circ g^q\bigr)(U)\;\cap\; int\bigl(f_1^r(W_r)\bigr)
        \;\neq\;\emptyset,
      $ which implies that $f_1^r\circ g^q(U) \cap f_1^r(W_r)\neq \emptyset$  
      then
      \(\;g^q(U)\cap W_r\neq\emptyset\)
      and hence \(g^q(U)\subseteq V_r\), for each $r \in \{1,...,p-1\}$. Thus $g^q(U)\subseteq V_r$ and applying \(f_1^r\) we get
      \[
        f_1^n(U)
        = f_1^r\bigl(g^q(U)\bigr)
        \;\subseteq\;
        f_1^r(V_r)
        \;\subseteq\;
        O.
      \]

In all cases \(f_1^n(U)\subseteq O\) whenever \(f_1^n(U)\cap V\neq\emptyset\).  Thus \((X,f_{1,\infty})\) is topologically equicontinuous, completing the proof.
\end{proof}

We now build some results to establish a topological analogue for Theorem \ref{T6}.

\begin{theorem}\label{HS}
Let \((X,f_{1,\infty})\) be a periodic non-autonomous dynamical system on a Hausdorff space $X$. Then $(X,f_{1,\infty})$ is Hausdorff sensitive if and only if $(X,g)$ is Hausdorff sensitive.

\end{theorem}

\begin{proof}
Suppose \((X,f_{1,\infty})\) is Hausdorff‐sensitive with sensitive finite open cover \(\mathcal U_0\).  For \(r\in \{1,\dots,p-1\}\) define
\[
\mathcal U_r
=\bigl\{\,f_1^{-r}(V):V\in\mathcal U_0\bigr\},
\]
which is again a finite open cover of \(X\).  Let
\(\displaystyle
\mathcal U
=\mathcal U_0\vee\mathcal U_1\vee\cdots\vee\mathcal U_{p-1}
\)
be their common refinement. Then $\mathcal{U}$ is a finite open cover for $X$. We claim that \(\mathcal U\) is a sensitive cover for \((X,g)\).

Assume if possible that \((X,g)\) is not Hausdorff sensitive then for the finite open cover $\mathcal{U}$, there exists a non-empty open set \(W\subseteq X\) such that for each \(n\in\mathbb N\), there exists a set $U_n\in\mathcal U$ such that $g^n(W)\;\subseteq\;U_n$. 

Let $k \in \mathbb{N}$, write \(k=qp+r\) with \(0\le r<p\). By our assumption there exists a $U_q$ such that $g^q(W)\subseteq U_q$, Since $\mathcal U
=\mathcal U_0\vee\mathcal U_1\vee\cdots\vee\mathcal U_{p-1}$, there is some $V_k \in\mathcal U_0$ with $U_q \subseteq f_1^{-r}(V_k).$ Hence $g^q(W)
\subseteq U_p\subseteq f_1^{-r}(V_r)$ which implies that
$f_1^r\bigl(g^k(W)\bigr)
\subseteq V_r.$ But \(f_1^r\circ g^q = f_1^{\,qp+r} = f_1^k\), therefore $f_1^k(W)\subseteq V_k \in \mathcal{U}_0$, hence for all $i \in \mathbb{N}, \ f_1^i(W) \subseteq V_i\in \mathcal{U}_0$ contradicting Hausdorff‐sensitivity of \((X,f_{1,\infty})\).  Thus \((X,g)\) is Hausdorff‐sensitive with sensitive cover \(\mathcal U\).\\
Now, suppose \((X,g)\) is Hausdorff sensitive with sensitive finite open cover \(\mathcal U\).  We show \((X,f_{1,\infty})\) is Hausdorff‐sensitive with the same cover. Let \(W\subseteq X\) be any non-empty open set.  By sensitivity of \(\mathcal U\) for \(g\), there exists an \(m\in\mathbb N\) such that $g^m(W)\;\not\subseteq\;U,
\ \text{for any }U\in\mathcal U.$ Let \(n = m\,p\).  Then $f_1^n(W) = f_1^{mp}(W)= g^m(W)$, so \(f_1^n(W)\not\subseteq U\), for any \(U\in\mathcal U\).  Hence \(\mathcal U\) is also a sensitive cover for \((X,f_{1,\infty})\).
\end{proof}

\begin{theorem}\label{STEq}
    Let \((X,f_{1,\infty})\) be a periodic non-autonomous dynamical system on a Hausdorff space $X$ such that  \(\mathrm{Tran}(X,g)\neq\emptyset\). If $ \mathrm{STEq}(X,f_{1,\infty)}\neq\emptyset $ then \(\mathrm{Tran}(X,f_{1,\infty})\subseteq \mathrm{STEq}(X,f_{1,\infty})\). Further, if \((X,f_{1,\infty})\) is minimal then it is syndetically topologically equicontinuous.
\end{theorem}
\begin{proof}
Assume that \(\mathrm{STEq}(X,f_{1,\infty})\neq\emptyset\).  Let
$
x \in \mathrm{Trans}(X,f_{1,\infty})
\text{ and }
z \in \mathrm{TEq}(X,f_{1,\infty}),
$
and fix any \(y \in X\) and neighborhood \(O\) of \(y\).  Since \((z,y) \in \mathrm{SEqP}(X,f_{1,\infty})\), there exist non-empty open sets $W\in \mathcal{N}_z$, $V'\in \mathcal{N}_y$ such that $
J_{f_{1,\infty}}(U,V;O)
$ is syndetic. By Theorem \ref{T1}, \(x \in \mathrm{Trans}(X,g)\).  Hence there is an \(N\in\mathbb N\) with
$
g^N(x) = f_1^{Np}(x) \in W.
$
Continuity of \(g^N\) gives a neighborhood \(U \in \mathcal{N}_x\) with
$
g^N(U) \subseteq W.
$
Then for each \(j\in J_{f_{1,\infty}}(W,V;O)\),
\[
f_1^{Np+j}(U)
= f_1^j\bigl(f_1^{Np}(U)\bigr)
= f_1^j\bigl(g^N(U)\bigr)
\subseteq f_1^j(W),
\]
therefore whenever \(f_1^{Np+j}(U)\cap V\neq\emptyset\) we have \(f_1^j(W)\cap V\neq\emptyset\) and hence \(f_1^{Np+j}(U)\subseteq O\), as \(f_1^{j}(W)\subseteq O\). Thus $Np + J_{f_{1,\infty}}(W,V;O)  \subseteq J_{f_{1,\infty}}(U,V;O)$ and since $J_{f_{1,\infty}}(W,V;O)$ is syndetic so is  $J_{f_{1,\infty}}(U,V;O)$ implying that \((x,y)\in \mathrm{SEqP}(X,f_{1,\infty})\). Since \(y\) and \(O\) were arbitrary, \(x \in \mathrm{STEq}(X,f_{1,\infty})\).  Therefore
\[
\mathrm{Trans}(X,f_{1,\infty})
\;\subseteq\;
\mathrm{STEq}(X,f_{1,\infty}).
\]
If \((X,f_{1,\infty})\) is minimal then \(\mathrm{Trans}(X,f_{1,\infty})=X\), implying \(\mathrm{STEq}(X,f_{1,\infty})=X\), i.e., the system is syndetically topologically equicontinuous.
\end{proof}

\begin{theorem}\label{THs}
Let \((X,f_{1,\infty})\) be a periodic non-autonomous dynamical system on a  $T_3$ space $X$ such that \(\mathrm{Tran}(X,g)\neq\emptyset\). Suppose that \(x\in \mathrm{Trans}(X,f_{1,\infty})\) . If there exists a \(y\in X\) with $(x,y)\notin \operatorname{SEqP}(X,f_{1,\infty})$, then the system \((X,f_{1,\infty})\) is thickly Hausdorff sensitive.
\end{theorem}

\begin{proof}
Suppose $p$ is the period of \((X,f_{1,\infty})\), then \(g:=f^{\,p}_{1}=f_{p}\circ f_{p-1}\circ\cdots\circ f_{1}\). We show that $(X,f_{1,\infty})$ is thickly Hausdorff sensitive. Let \(x\) and \(y\) be as in the hypothesis. Then by Theorem \ref{T1}, $x\in \mathrm{Trans}(X,g)$. Since $(x,y)\notin \operatorname{SEqP}(X,f_{1,\infty})$, there exists a neighborhood \(O\in \mathcal{N}_y\) such that for every choice of neighborhoods \(U\in \mathcal{N}_x\) and \(V\in \mathcal{N}_y\), $J_{f_{1,\infty}}(U,V;O)$ is not syndetic and hence $\mathbb N \setminus J_{f_{1,\infty}}(U,V;O)$ is thick. So for any $L \in \mathbb N$, there exist $n_0$ such that for each $t\in \{n_0,n_0+1,....,n_0+L\}$, we have $f_1^t(U)\cap V\neq\emptyset \quad\text{and}\quad f_1^t(U)\not\subseteq O$.

Since \(X\) is $T_3$, we can choose open neighborhoods \(V_1\) and \(V_2\) of \(y\) satisfying $\overline{V}_1\subseteq O \quad \text{and} \quad \overline{V}_2\subseteq V_1$. Consider the finite open cover $\mathcal{U}=\{V_1,\, X\setminus \overline{V_2} \}$.
Now, let \(W\subseteq X\) be an arbitrary non-empty open set and $l\in \mathbb N$. Since \(x\in \mathrm{Trans}(X,f_{1,\infty})=\mathrm{Trans}(X,g)\), there exists an \(n\in\mathbb{N}\) such that $x \in W'= g^{-n}(W)$ and since $(x,y)\notin \operatorname{SEqP} (X,f_{1,\infty})$, by Lemma \ref{ls}, there exists an integer \(s>pn\) such that for $j\in\{s,s+1,...s+l\}$ we have $f_1^j(W')\cap V_2\neq\emptyset \quad\text{and}\quad f_1^j(W')\not\subseteq O$. Since $W' = g^{-n}(W)=(f_{1}^{pn})^{-1}(W)$, we have $ f_1^j(W') = f_1^{j-pn}(W)$, therefore, for \(k=j-pn\), we get that $f_1^k(W)\cap V_2\neq\emptyset \quad\text{and}\quad f_1^k(W)\not\subseteq O$. In particular, there exist  \(a,b\in W\) such that $f_1^k(a)\notin O  \text{ and } f_1^k(b)\in V_2$. Thus,
\[
\{f_1^k(a),f_1^k(b)\}\cap V_1=\{f_1^k(b)\} \quad \text{and} \quad \{f_1^k(a),f_1^k(b)\}\cap (X\setminus \overline{V_2})=\{f_1^k(a)\},
\]
which shows that the images of \(a\) and \(b\) lie in different members of the cover \(\mathcal{U}\) and hence $k \in N_{f_{1,\infty}}(W,\mathcal{U}).$ Therefore $\{s-pn, (s-pn)+1,...(s-pn)+l\}\subseteq N_{f_{1,\infty}}(W,\mathcal{U})$ and hence $N_{f_{1,\infty}}(W,\mathcal{U})$ is thick. Since \(W\) was arbitrary, it follows that \((X,f_{1,\infty})\) is thickly Hausdorff sensitive.
\end{proof}

We now have the following dichotomy that is analogous to Theorem \ref{T6}.

\begin{theorem}\label{TAY2}
Let \((X,f_{1,\infty})\) be a periodic non-autonomous dynamical system on a  $T_3$ space $X$ such \(\mathrm{Tran}(X,g)\neq\emptyset\). If \((X,f_{1,\infty})\) is minimal then it is either syndetically topologically equicontinuous or thickly Hausdorff sensitive.
\end{theorem}

\begin{proof}
Suppose \((X,f_{1,\infty})\) is minimal. There are two possibilities:

\medskip
\noindent\textbf{Case 1:} \(\mathrm{STEq}(X,f_{1,\infty})\neq \emptyset\). Then by Theorem \ref{STEq}, \(\mathrm{STEq}(X,f_{1,\infty})=X\). Hence, the system \((X,f_{1,\infty})\) is syndetically topologically equicontinuous.

\medskip
\noindent\textbf{Case 2:} Suppose that no point is topologically equicontinuous. In this case, for every \(x\in X\) there exists some \(y\in X\) such that 
$
(x,y)\notin \operatorname{SEqP}(X,f_{1,\infty}).
$
Then by the Theorem \ref{THs}, the system \((X,f_{1,\infty})\) is thickly Hausdorff sensitive. Hence, a minimal non-autonomous dynamical system \((X,f_{1,\infty})\) must be either syndetically topologically equicontinuous or thickly Hausdorff sensitive.
\end{proof}

Theorem \ref{HTS} shows that a periodic non-autonomous dynamcial system is thickly Hausdorff sensitive if and only if the induced autonomous system is Hausdorff sensitive.
\begin{theorem}\label{HTS}
    Let \((X,f_{1,\infty})\) be a periodic non-autonomous dynamical system on a Hausdorff space $X$ and let \((X,g)\) be its induced autonomous system. Then $(X,f_{1,\infty})$ is thickly Hausdorff sensitive if and only if $(X,g)$ is thickly Hausdorff sensitive.
\end{theorem}
\begin{proof}
    Suppose $p$ is the period of \((X,f_{1,\infty})\), then \(g:=f^{\,p}_{1}=f_{p}\circ f_{p-1}\circ\cdots\circ f_{1}\). We first show that if $(X,g)$ is thickly sensitive then \((X,f_{1,\infty})\) is thickly sensitive.
Suppose \((X,g)\) is thickly sensitive with thick sensitive finite open cover \(\mathcal U_0\) and assume that  \((X,f_{1,\infty})\) is not thickly sensitive.  For \(r\in \{1,\dots,p-1\}\) define
\[
\mathcal U_r
=\bigl\{\,(f_{p}\circ f_{p-1}\circ\cdots\circ f_{r+1})^{-1}(V)=(f_{r+1}^{p-r})^{-1}(V):V\in\mathcal U_0\bigr\},
\]
which is again a finite open cover of \(X\).  Let
\(\displaystyle
\mathcal U
=\mathcal U_0\vee\mathcal U_1\vee\cdots\vee\mathcal U_{p-1}
\)
be their common refinement. Then $\mathcal{U}$ is a finite open cover for $X$. Since  \((X,f_{1,\infty})\) is not thickly Hausdorff sensitive, for $\mathcal{U}$, there exists a non-empty open set $W$ of $X$ such that $N_{f_{1,\infty}}(W,\mathcal{U})$ is not thick and therefore $\mathbb N \setminus N_{f_{1,\infty}}(W,\mathcal{U})$ is syndetic, implying that there exists an $L\in \mathbb N$ such that for each $n\in \mathbb N$, there exists an $i\in \{n,n+1,...,n+L\}$ such that $f_1^i(W)\subseteq U$ for some $U\in \mathcal{U}$.\\
Since $(X,g)$ is thickly Hausdorff sensitive, $N_g(W,\mathcal{U}_0)$ is thick. Choose $k\in\mathbb N$ with $kp>L$, then there exists an $n_0\in \mathbb N$ such that 
\[ \{ n_0, n_0+1,...,n_0+k\} \subseteq N_g(W,\mathcal{U}_0). \]\\
Consider the block $ \{ n_0p, n_0p+1,...,n_0p+L\}$, then there exists a $j\in \{ n_0p, n_0p+1,...,n_0p+L\}$ such that \[f_1^j(W)\subseteq U=\bigl (U_0\; \cap \; \bigcap_{r=1}^{p-1}(f_{r+1}^{p-r})^{-1}(U_r) \bigr )\]
Let $j=pq+r$, with $\ 0\leq r<p $ and $n_0\leq q \leq (n_0+k-1)$. If $j$ is a multiple of $p$, then $j=pq$ and $q\in N_{f_{1,\infty}}(W,\mathcal{U_0})$ therefore $f_1^j(W)=g^q(W) \subseteq U_0$ which is a contradiction.\\
If $r \not= 0$ then $f_1^j(W) \subseteq U\subseteq (f_{p}\circ f_{p-1}\circ\cdots\circ f_{r+1})^{-1}(U_r)$, therefore $g^{q+1}(W)=f_{p}\circ f_{p-1}\circ\cdots\circ f_{r+1}(f_1^r(g^q(W)))\subseteq Ur$, which is a contradiction as $q+1\in \{ n_0, n_0+1,...,n_0+k\}\subseteq N_g(W,\mathcal{U}_0) $.\\
Hence, if $(X,g)$ is thickly Hausdorff sensitive then $(X,f_{1,\infty})$ is thickly Hausdorff sensitive.\\
Now Assume that \((X,f_{1,\infty})\) is thickly Hausdorff sensitive with a finite open cover \(\mathcal{D}\). We will show that \((X,g)\) is thickly sensitive with the same \(\mathcal{D}\). Let \(V\subseteq X\) be any non-empty open set and fix an arbitrary \(t\in\mathbb N\). Choose   \(k\in\mathbb{N}\) with $k>(t+1)\,p$. Since \((X,f_{1,\infty})\) is thickly Hausdorff sensitive, there exists an $m_0\in \mathbb N$ such that
$
\{\,m_0,\;m_0+1,\;\dots,\;m_0+k\}
\;\subseteq\;
N_{f_{1.\infty}}(V,\mathcal{D})
$
. We can write
$
m_0 \;=\; q\,p \;+\; r,
\
0 \;\le\; r < p
$
, then $\{(q+1)p,\;(q+2)p,\;\dots,\;(q+1+t)p\}\subseteq N_{f_{1.\infty}}(V,\mathcal{D})$. But $f_1^{\,(q+j)p} \;=\; g^{\,q+j}$, so $\{\,q+1,\;q+2,\;\dots,\;q+1+t\}
\;\subseteq\;
N_g(V,\mathcal{D})$.

Since \(t\) and $V$ were arbitrary, therefore \((X,g)\) is thickly Hausdorff sensitive with sensitive open cover \(\mathcal D\).

\end{proof}
Theorem \ref{hmsts} is analogous to \cite[Theorem 3.2]{huang2018analogues} and provides sufficient conditions under which thick Hausdorff sensitivity and multi-Hausdorff sensitivity are equivalent for autonomous dynamical systems. We use it to establish a similar result for periodic non-autonomous dynamical systems in Theorem \ref{Tpms}.

\begin{theorem}\label{hmsts}
    Let \((X,f)\) be an autonomous system on a Hausdorff space \(X\). If \((X,f)\) is multi-Hausdorff sensitive then it is thickly Hausdorff sensitive. If $\mathrm{Trans}(X,f)\neq\emptyset$, then thick Hausdorff sensitivity of \((X,f)\) implies multi-Hausdorff sensitivity.
\end{theorem}
\begin{proof}
    First Suppose that $(X,f)$ is  multi-Hausdorff sensitive with the multi-sensitivity finite cover $\mathcal{U}$. Let $k \in \mathbb{N}$ and $U$ be a non-empty open subset of $X$. Let $i \in \{1,2,...,k\}$, then $f^{-i}(U)\neq \emptyset$. Suppose $x_i\in f^{-i}(U)$ and $f^{i}(x_i) \in V_i$, where $V_i\in \mathcal{U}$. By continuity of $f^i$, there exists an open set $U_i\subseteq X$ such that $f^i(U_i)\in V_i$. Let $s \in \bigcap_{i=1}^k N_f(U_i, \mathcal{U})$ then by construction $s > k$ and $s \in \bigcap_{i=1}^k N_f(f^{-i}(U), \mathcal{U})$ and therefore $\{s-k,s-k+1,..., s-1,s\} \subseteq N_f(U, \mathcal{U})$. This shows $N_f(U, \mathcal{U})$ is thick and $(X,f)$ is thickly Hausdorff sensitive.

    Now assume that $(X,f)$ is thickly Hausdorff sensitive  with thick sensitivity finite cover $\mathcal{U}$ and $\mathrm{Trans}(X,f)\neq\emptyset$. Let $m\in \mathbb N$ and $V_1,V_2...,V_m$ be a collection of non-empty open subsets of $X$. Suppose $x\in\mathrm{Tran}(X,f)$, then there exists an $n_i\in \mathbb N$ such that $f^{n_i}(x)\in V_i$ for each $i\in \{1,2,...,m\}$. By continuity of $f$, there exists a $V\in\mathcal{N}_x$ such that $f^i(V)\subseteq V_i$ for each $i\in \{1,2,...,m\}$. By our assumption $(X,f)$ is thickly Hausdorff sensitive and therefore there exists a $t\in \mathbb N$ such that  $\{t, t+1,...,t+n_1+n_2+...+n_m\}\subseteq N_f(V, \mathcal{U}) $ and hence $t \in \bigcap_{i=1}^m N_f(V_i, \mathcal{U})$. Thus $(X,f)$ is multi-Hausdorff sensitive.
\end{proof}

\begin{theorem}\label{Tpms}
Let \((X,f_{1,\infty})\) be a periodic nonautonomous system on a Hausdorff space \(X\) such that $\mathrm{Tran}(X,g)\neq\emptyset$. Then \((X,f_{1,\infty})\) is thickly Hausdorff sensitive if and only if \((X,f_{1,\infty})\) is multi-Hausdorff sensitive.
\end{theorem}
\begin{proof}
First assume that \((X,f_{1,\infty})\) is multi-Hausdorff sensitive, then as shown in Theorem \ref{HS}, $(X,g)$ is multi-Hausdorff sensitive and therefore by Theorem \ref{hmsts}, $(X,g)$ is thickly Hausdorff sensitive and hence by Theorem \ref{HTS} \((X,f_{1,\infty})\) is thickly Hausdorff sensitive.\\
Now, assume that \((X,f_{1,\infty})\) is thickly Hausdorff sensitive, then by Theorem \ref{HTS} $(X,g)$ is thickly sensitive and by Theorem \ref{hmsts}, $(X,g)$ is multi-Hausdorff sensitive and hence \((X,f_{1,\infty})\) is multi-Hausdorff sensitive.
\end{proof}

\begin{corollary}
    Let \((X,f_{1,\infty})\) be a periodic non-autonomous dynamical system on a $T_3$ space $X$ such that \(\mathrm{Tran}(X,g)\neq\emptyset\). If \((X,f_{1,\infty})\) is minimal then it is either syndetically topologically equicontinuous or multi-Hausdorff sensitive.
\end{corollary}

\section*{Acknowledgment}
The first author is supported by Junior Research Fellowship, E-certificate No. $24J/01/00293$ and CSIR-HRDG Ref. No: June-24(ii)/EU-V, for carrying out doctoral work.

\end{document}